\documentclass[11pt, reqno]{amsart}
\usepackage{euscript,amscd,amsgen,amsfonts,amssymb,latexsym,graphicx,psfrag}

\newcommand{\eqdef}{\stackrel{\scriptscriptstyle\rm def}{=}}

\newtheorem{theorem}{Theorem}

\newtheorem{proposition}{Proposition}
\newtheorem{corollary}{Corollary}

\newtheorem{lemma}{Lemma}
\newtheorem*{remark}{Remark}
\newtheorem*{remarks}{Remarks}
\newtheorem{example}{Example}

\newcommand{\beha}{\begin{enumerate}}
\newcommand{\behe}{\end{enumerate}}
\renewcommand{\epsilon}{\varepsilon}

\newcommand{\R}{{\rm Rot}}

\newcommand{\Or}{\mathcal{O}}
\newcommand{\C}{\mathcal{C}}

\newcommand{\cM}{\EuScript{M}}

\newcommand{\cH}{\EuScript{H}}
\newcommand{\bR}{{\mathbb R}}

\newcommand{\bZ}{{\mathbb Z}}
\newcommand{\bN}{{\mathbb N}}

\newcommand{\cA}{{\mathcal A}}
\newcommand{\cB}{{\mathcal B}}
\newcommand{\cC}{{\mathcal C}}

 \def\car{{\rm card}}
\def\1{1\!\!1}

\def\and{\text{ and }}

\def\h{{\text h}}

                        \def\^{\tilde}

\def\Per{{\rm Per}}

\def\Fix{{\rm Fix}}
\def\Fix{{\rm Fix}}
\def\Fix{{\rm Fix}}
\def\Per{{\rm Per}}
\def\EPer{{\rm EPer}}

\def\STP{{\rm{(STP)}}}

\def\1{1\!\!1}

\def\rv{{\rm rv}}

\newtheorem*{thmA}{Theorem A}

\newtheorem*{thmB}{Theorem B}

\DeclareMathSymbol{\varnothing}{\mathord}{AMSb}{"3F}
\renewcommand{\emptyset}{\varnothing}
\title{Geometry and entropy of generalized rotation sets}
\author{Tamara Kucherenko}\address{Department of Mathematics, The City College of New York, New York, NY, 10031, USA}\email{tkucherenko@ccny.cuny.edu}

\author{Christian Wolf}\address{Department of Mathematics, The City College of New York, New York, NY, 10031, USA}\email{cwolf@ccny.cuny.edu}

\thanks{This work was partially supported by a grant from the Simons Foundation (\#209846 to Christian Wolf).}

\begin{document}

\begin{abstract}

 For a continuous map $f$ on a compact metric space we study the geometry and entropy of the generalized  rotation set $\R(\Phi)$. Here $\Phi=(\phi_1,...,\phi_m)$ is a $m$-dimensional continuous potential and $\R(\Phi)$ is the set of all $\mu$-integrals of $\Phi$ and $\mu$ runs over all $f$-invariant probability measures. It is easy to see that the rotation set is a compact and convex subset of $\bR^m$.  We study the question if every compact and convex set is attained as a rotation set of a particular set of potentials within a particular class of dynamical systems. We give a positive answer in the case of subshifts of finite type by constructing for every compact and convex set $K$ in $\bR^m$ a potential $\Phi=\Phi(K)$ with $\R(\Phi)=K$.
Next, we study the relation between $\R(\Phi)$ and the set of all statistical limits $\R_{Pt}(\Phi)$. We show that in general these sets differ but also provide criteria that guarantee $\R(\Phi)= \R_{Pt}(\Phi)$. Finally,
we study the entropy function $w\mapsto H(w), w\in \R(\Phi)$. We establish a variational principle for the entropy function and show that for certain non-uniformly hyperbolic systems $H(w)$ is determined by the growth rate of those hyperbolic periodic orbits whose $\Phi$-integrals are close to $w$. We also show that for systems
with strong thermodynamic properties (subshifts of finite type, hyperbolic systems and expansive homeomorphisms with specification, etc.) the entropy function $w\mapsto H(w)$ is real-analytic in the interior of the rotation set.

\end{abstract}
\subjclass[2000]{}
\maketitle

%---------------------------------------------------------------------------
\section{Introduction}
%---------------------------------------------------------------------------

{\subsection{ Motivation. }
It frequently occurs in applications that, in the presence of an underlying dynamical system,
 only limited knowledge is available about the actual time evolution of a particular state of the system rather than information
about  certain  measurements along its orbit. It is then a natural problem to determine what asymptotic measurements one is able to observe for a particular class of systems
and what kind of information about the underlying system can be recovered from these measurements.
In this paper we consider deterministic discrete-time dynamical systems given by a continuous map $f:X\to X$ on a compact metric space $X$. We will deal with finitely many measurements
that are given by a continuous $m$-dimensional potential $\Phi=(\phi_1,\cdots,\phi_m):X\to \bR^m$. To recover information
about the "typical" dynamics of $f$ let us consider the set $\cM$ of all Borel invariant probability measures and denote by $\cM_E\subset \cM$ the subset of ergodic measures. By Birkhoff's Ergodic Theorem, for each $\mu\in \cM_E$ there exists a set $\cB(\mu)$ of full $\mu$-measure  (called the {\it basin } of $\mu$) such that the Birkhoff averages $\frac{1}{n}S_n\phi(x)$, where
\begin{equation}\label{eqsn}
S_n \phi(x)=\sum_{k=0}^{n-1} \phi(f^k(x)),
\end{equation}
converge to $\int \phi\ d\mu$ as $n\to\infty$ for all $x\in \cB(\mu)$ and all $\phi\in C(X,\bR)$. Following \cite{GM, Je}
 we call   $\R(\Phi)=\R(f,\Phi)$ defined by
\begin{equation} \label{defrotset}
 \R(\Phi)= \left\{\rv(\mu): \mu\in\cM\right\},
\end{equation}
 the {\it generalized rotation set} of  $\Phi$ with respect to $f$,
where
\begin{equation}
\rv(\mu)=\left(\int \phi_1\ d\mu,\cdots,\int \phi_m\ d\mu\right)
\end{equation}
 denotes the rotation vector of the measure $\mu$. This terminology goes back to Poincar\'e's rotation numbers for circle homeomorphisms \cite{Po}. Generalization of rotation numbers to higher dimensional tori leads to studying limits of Birkhoff averages $\frac{1}{n}S_n \phi(x)$ for a displacement function $\phi$ (\cite{GM},\cite{KMG}).
More generally for a potential $\phi$
one can vary $x$ and consider limits of statistical averages $\frac{1}{n_l}S_{n_l} \phi(x_l)$, where $(x_l)\subset X$ and $n_l\to\infty$. The set of all such limits is referred to as {\it generalized pointwise rotation set} of $\phi$ and we will denoted it by $\R_{Pt}(\phi)$. We refer to the overview article \cite{Mi1} and references therein for
further details about rotation sets.

One way that indicates the relevance of the rotation set for the understanding of the typical behavior of the dynamical system is to consider a sequence of potentials $(\phi_k)_{k\in \bN}$ that is dense in   $C(X,\bR)$. Let ${\rm R}_m$ be the rotation set of the initial $m$-segment of potentials, that is ${\rm R}_m= \R(\phi_1,\cdots,\phi_m)$. It follows from the representation theorem that the rotation classes of the rotation sets ${\rm R}_m$
form a  decreasing sequence of  covers of $\cM$ whose intersections contain a unique  invariant measure.
Therefore, for large $m$ the set ${\rm R}_m$ provides a fine cover of $\cM$ and can  be considered as a finite dimensional approximation of the set of all invariant probability measures.

In this paper we consider various classes of dynamical systems and potentials and study the geometric possibilities for the associated rotation sets. Moreover, we  consider a notion of entropy associated with rotation vectors  and study several aspects of this entropy including variational properties, relation to periodic orbits and regularity of the rotation entropy map.

We will now describe our results in more detail.

\subsection{Statement of the Results. }
Let $f:X\to X$ be a continuous map on a compact metric space $X$, let  $\Phi\in C(X,\bR^m)$ be a continuous ($m$-dimensional) potential, and let $\R(\Phi)$ be the rotation set of $\Phi$. It follows  from the definitions that $\R(\Phi)$ is a compact and convex subset of $\bR^m$. In particular, $\R(\Phi)$ has a Lipschitz boundary and therefore $\partial \R(\Phi)$ is differentiable at $\cH^{m-1}$-almost every boundary point, where $\cH^{m-1}$ is the $(m-1)$-dimensional Hausdorff measure.

Considerable attention has been given to the question which compact convex sets can arise as rotation sets. It is shown by Kwapisz \cite{Kw1, Kw2} that any polygon whose vertices are at rational points in the plane is a classical rotation set of some homeomorphism on the two-torus, however such rotation sets are not necessarily polygons. He also proved that certain line segments cannot be realized as such rotation sets \cite{Kw3}. Ziemian studied the case when a dynamical system is a transitive subshift of finite type and the potential $\Phi$ is constant on cylinders of length two \cite{Z}. She  showed that under these assumptions $\R(\Phi)$ is a polyhedron. In contrast to this restrictive geometry  for rotation sets we have the following
result (see Theorems \ref{thm1} and \ref{thm2m} in the text).

\begin{thmA} Let $K\subset \bR^m$ be compact and convex. Then there exists a one-sided full shift $f:X\to X$ on a shift space with a finite alphabet and a $m$-dimensional continuous potential
$\Phi$ such that $\R(\Phi)=K$.
\end{thmA}

The main idea in the proof of Theorem A is to approximate $K$ from inside by polytopes $\mathcal{P}_j$ and  to construct potentials $\Phi_j$ with $\R(\Phi_j)=\mathcal{P}_j$ in such a way that $\Phi_j$ converges
uniformly to a continuous potential $\Phi$ with $\R(\Phi)=K$. While our result is formulated for one-sided full shifts it can be easily generalized to (one-sided and two-sided) mixing subshifts of finite type, hyperbolic systems
and expansive homeomorphisms with specification. As a consequence of Theorem A we obtain that for shift maps it is possible that the boundary of the rotation set
is non-differentiable at a countable dense set of boundary points. We note that the potential $\Phi$ constructed in Theorem A is in general not H\"older
continuous. In Corollary \ref{corende} we construct for every H\"older continuous potential a natural family of analytic hypersurfaces in ${\rm int }\ K$  associated with certain equilibrium measures  that converge to $\partial K$ with respect to the Hausdorff metric. This could make one believe that for shift maps and H\"older continuous potentials the boundary of the rotation set is at least piece-wise smooth. However, this
is in general not true due to the example by Bousch, see \cite{B, Je2}.

Next, we discuss closely related concepts of rotation sets. Let $\R_{Pt}(\Phi)$ be the set of all $w\in \bR^m$ which are accumulation points of Birkhoff averages (see \eqref{defRf}
for the precise definition). The set $\R_{Pt}(\Phi)$ is frequently called the {\it pointwise rotation set} in the literature (see for example \cite{GM}).
Moreover, let $\R_E(\Phi)= \left\{\rv(\mu): \mu\in\cM_E\right\}$. It follows from
 Proposition \ref{prop456} and Examples 1 and 2 that for any continuous
map $f:X\to X$ and any $\Phi\in C(X,\bR^m)$ we have
\[
\R_E(\Phi)\subset \R_{Pt}(\Phi)\subset \R(\Phi),
\]
and in general each of these inclusions  can be strict. We also give conditions implying $ \R_{Pt}(\Phi)= \R(\Phi)$ which are satisfied for various classes of systems.

A natural invariant that quantifies the dynamical complexity of an invariant measure  is the measure-theoretic entropy of $\mu$ denoted by $\h_\mu(f)$ (see for example~\cite{Wal:81} for details). Following \cite{Je} we define the entropy of $w\in \R(\Phi)$ by
\begin{equation}\label{defH}
H(w)\eqdef\sup\{h_\mu(f): \mu \in \cM\ {\rm and }\ \rv(\mu)=w\}.
\end{equation}
The entropy of rotation vectors was extensively studied by Geller and Misiurewicz in \cite{GM} as well as by Jenkinson, who added fundamental contributions in \cite{Je, Je1, Je2}. An alternative definition of entropy (denoted by $h(w)$), which is closely related to that of topological entropy,  is in terms of the exponential growth rate of the cardinality of maximal $(n,\epsilon)$-separated sets of points whose Birkhoff averages are "close" to that of $w$.
We refer to \eqref{eqdefhw234} for the precise definition. We obtain that $h(w)\leq H(w)$ holds for all $w$ at which $H$ is continuous and $h(w)=H(w)$ holds for those $w\in \R_{Pt}(\Phi)$ whose entropy can be approximated by
ergodic measures (see Theorem \ref{thhwHw}). This result can be interpreted as a conditional variational principle with the condition being only using points and measures having rotation vectors
close respectively equal to $w$.

It is known since the classical works of Bowen that in framework of symbolic and hyperbolic dynamics entropy can be computed in terms of growth rate of periodic orbits rather than using arbitrary $(n,\epsilon)$-separated sets.
In the papers \cite{GW1,GW2} Gelfert and Wolf extended these results to smooth dynamical systems exhibiting some non-uniformly hyperbolic behavior. In Theorem  \ref{theoperrot} we apply  techniques from \cite{GW2} to compute $H(w)$ in terms of the growth rate of certain hyperbolic periodic orbits.

Our last result deals with systems which have strong thermodynamic properties, i.e. shift maps, uniformly hyperbolic systems and expansive homeomorphisms with specification.
Since for these systems the entropy map $\mu\mapsto h_\mu(f)$ is upper-semi continuous, the rotation entropy map $w\mapsto H(w)$ is continuous. However, we are able to obtain a stronger regularity.

\begin{thmB}
 Suppose $f:X\to X$ has strong thermodynamic properties and assume $\Phi:X\to\bR^m$ is H\"older continuous.
Then $w\mapsto H(w)$ is real-analytic on  the interior of $ \R(\Phi)$. Moreover, $H(w)>0$ for all $w\in {\rm int}\ \R(\Phi)$.
\end{thmB}

The proof of this theorem   relies heavily on methods from the thermodynamic formalism and, in particular, on the analyticity
of the topological pressure for H\"older continuous potentials. Moreover, we use results of Jenkinson \cite{Je} as  key ingredients.

This paper is organized as follows. In Section 2 we review some
basic concepts and results  about symbolic and smooth dynamics and the thermodynamic formalism.
Section 3 is devoted to the proof of    Theorem A.
 In Section 4 we introduce different notions of rotation sets and rotation entropy and study their relations.
In Section 5 we consider non-uniformly hyperbolic dynamics and derive results about the relation between rotation entropy and the growth rates of periodic orbits.
Finally, we Section 6 the dependence of $H(w)$ on $w$ for systems with strong thermodynamic properties is studied.

%---------------------------------------------------------------------------
\section{Preliminaries}\label{sec:2}
%-------------------------------------------------------------------------
In this section we discuss relevant background material which will be used later on. We will continue to use the notations from Section 1.
Let $f:X\to X$ be a continuous map on a compact metric space $(X,d)$. Let $\Phi=(\phi_1,\cdots,\phi_m)\in C(X,\bR^m)$ and let $\R(f,\Phi)$
be the rotation set of $f$ with respect to $\Phi$. Since we will always work with a fixed dynamical system $f$ we frequently omit the dependence on $f$ in the notation of the rotation set and write $\R(\Phi)$ instead of $\R(f,\Phi)$.
Let $\cM$ denote the space of all $f$-invariant Borel probability measures on $X$ endowed with the weak$\ast$ topology. This makes $\cM$ a compact convex space. Moreover, let $\cM_E\subset \cM$ be the subset of ergodic measures.
Throughout the entire paper we use as a standing assumption
that $f$ has finite topological entropy.
\subsection{Rotation vectors of  periodic points. }

We denote by $\Fix(f)$ the set of all fixed points of $f$, by $\Per(f)$  the set of all periodic points of $f$ and by $\Per_n(f)$ the set of all $x\in \Fix(f^n)$.

Given $x\in \Per_n(f)$ we denote by $\mu_x$ the unique invariant probability measure  supported on the orbit of $x$, that is,
\begin{equation}\label{defpermes}
\mu_x=\frac{1}{n}\sum_{k=0}^{n-1} \delta_{f^k(x)},
\end{equation}
where $\delta$ denotes the delta Dirac measure supported on the point. Moreover, we define the rotation vector of $x$ by
\begin{equation}
\rv(x)\eqdef \rv(\mu_x) = \frac{1}{n} \sum_{k=0}^{n-1} \Phi(f^k(x)).
\end{equation}
Given $w\in \R(\Phi), n\in \bN$ and $r>0$ we write
\begin{equation}
\Per_n(w,r)=\{x\in \Per_n(f): \rv(x)\in D(w,r)\}
\end{equation}
and $\Per(w,r)=\bigcup_{n\in \bN} \Per_n(w,r)$. Here $D(w,r)$ denotes the open  ball about $x\in \bR^m$ and radius $r$ with respect to the Euclidean norm.

\subsection{Thermodynamic formalism}
Let us first recall the notion of topological pressure. Let
$(X,d)$ be a compact metric space and let $f\colon
X\to X$ be a continuous map. For $n \in {\mathbb N}$ we
define a new metric $ d_n $ on $ X $ by $
d_n(x,y)=\max_{k=0,\ldots ,n-1} d(f^k(x),f^k(y))$. The balls in the $d_n$ metric are frequently called Bowen balls. A set of points
$\{ x_i\colon i\in I \}\subset X$ is called
\emph{$(n,\varepsilon)$-separated} (with respect to $f$) if
$d_n(x_i,x_j)> \varepsilon$ holds for all $x_i,x_j$ with $x_i \ne
x_j$. Fix for all $\varepsilon>0$ and all $n\in\bN$ a maximal
(with respect to the inclusion) $(n,\varepsilon)$-separated set
$F_n(\epsilon)$. The \emph{topological pressure} (with respect to $f$) is a mapping
$ P_{\rm top}(f,\cdot)\colon C(X,\bR)\to \bR$  defined by
\begin{equation}\label{defdru}
  P_{\rm top}(f,\phi) = \lim_{\varepsilon \to 0}
            \limsup_{n\to \infty}
            \frac{1}{n} \log \left(\sum_{x\in F_n(\epsilon)}
            \exp S_n\phi(x) \right),
\end{equation}
where $S_n\phi(x)$ is defined as in \eqref{eqsn}.
Recall that the topological entropy of $f$ is defined by
$h_{\rm top}(f)=P_{\rm top}(f,0)$. We simply write $P_{\rm
  top}(\phi)$ and $h_{\rm top}$ if there is no confusion about $f$.
Note that the definition of $P_{\rm top}(\phi)$ does not depend on the choice of the
sets $F_n(\epsilon)$ (see~\cite{Wal:81}).
The topological pressure satisfies the
well-known variational principle, namely,
\begin{equation}\label{eqvarpri}
P_{\rm top}(\phi)=
\sup_{\mu\in \cM} \left(h_\mu(f)+\int_X \phi\,d\mu\right).
\end{equation}
Here $h_\mu(f)$ denotes the measure-theoretic entropy of $f$ with respect to $\mu$ (see~\cite{Wal:81} for details).
It is easy to see that the supremum in~\eqref{eqvarpri} can be replaced by
the supremum taken only over all $\mu\in\cM_{\rm E}$.

 If there exists a measure
$\mu\in \cM$ at which the supremum in \eqref{eqvarpri} is
attained it is called an equilibrium state (or also equilibrium measure) of the potential $\phi$.
We denote by $ES(\phi)$ the set of all equilibrium states of $\phi$.
In general $ES(\phi)$ may be empty.
Note that if the entropy map
\begin{equation}\label{entup}
\mu\mapsto h_{\mu}(f)
\end{equation}
is upper semi-continuous on $\cM$ then for each $\phi\in C(X,\bR)$ we have that $ES(\phi)\not=\emptyset$. Since $ES(\phi)$ is a compact, convex subset of $\cM$ whose extremal points are the ergodic measures (see \cite{Wal:81}), we obtain in this case that
\begin{equation}\label{esne}
ES(\phi)\cap \cM_E\not=\emptyset
\end{equation}
for all $\phi\in C(X,\bR)$.

Given $\alpha\in(0,1]$, let
$C^\alpha(X,\bR)$ be the space of H\"older continuous functions
with H\"older exponent~$\alpha$.
We
recall that two functions $\phi$, $\psi\colon X\to\bR$ are
said to be cohomologous if $\phi-\psi=\eta-\eta\circ f$ for
some continuous function $\eta\colon X\to\bR$.

We now list several properties of the topological pressure which hold for certain classes of dynamical systems which we will
be discussed latter on.

 We say $f:X\to X$ has strong thermodynamic properties (which we abbreviate by \STP) if the following conditions hold:
\begin{enumerate}
\item $h_{\rm top}(f)<\infty$;
\item  The entropy map $\mu\mapsto h_\mu (f)$ is upper semi-continuous;
\item  The map $\phi\mapsto  P_{\rm top}(f,\phi)$ is real-analytic on $C^\alpha(X,\bR)$;
\item \label{ref4} Each  potential $\phi \in C^\alpha(X,\bR)$ has a
unique equilibrium measure $\mu_\phi\in ES(\phi)$. Furthermore,
$\mu_\phi$ is ergodic and given $\psi\in C^\alpha(X,\bR)$ we have
\begin{equation}\label{eqdifpre}
\frac{d}{dt} P_{\rm top}(f,\phi + t\psi )\Big|_{t=0}= \int_X \psi
\,d\mu_\phi.
\end{equation}
\item  For each $\phi$, $\psi\in C^\alpha(X,\bR)$ we have
$\mu_\phi=\mu_\psi$ if and only if $\phi-\psi$ is cohomologous
to a constant.
\item \label{ref5} For each $\phi$, $\psi\in C^\alpha(X,\bR)$ and
$t\in\bR$ we have
\begin{equation}\label{gg33}
\frac{d^2}{dt^2} P_{\rm top}(f,\phi + t\psi )\ge 0,
\end{equation}
with equality if and only if $\psi$ is cohomologous to a constant.

\end{enumerate}
Note that for several classes of systems properties (3)-(6) hold even for a wider class of potentials, namely for potentials with summable variation (see for example \cite{Je}). For simplicity, we restrict our
considerations to H\"older continuous potentials.

Next we discuss some examples with strong thermodynamic properties.
\subsection{Shifts and subshifts}

Let $d\in \bN$ and let $\cA=\{0,\cdots,d-1\}$ be a finite alphabet in $d$ symbols. The (one-sided) shift space $X$ on the alphabet $\cA$ is the set of
all sequences $x=(x_n)_{n=1}^\infty$ where $x_n\in \cA$ for all $n\in \bN$.  We endow $X$ with the Tychonov product topology
which makes $X$ a compact metrizable space. For example, given $0<\alpha<1$ it is easy to see that
\begin{equation}\label{defmet}
d(x,y)=d_\alpha(x,y)=\alpha^{\inf\{n\in \bN:\  x_n\not=y_n\}}
\end{equation}
defines a metric which induces the Tychonov product topology on $X$.
The shift map $f:X\to X$ (defined by $f(x)_n=x_{n+1}$) is a continuous $d$ to $1$ map on $X$.
If $Y\subset X$ is a $f$-invariant set we  say  that $f|_Y$ is a sub-shift. In particular, for  a $d\times d$ matrix $A$ with values in $\{0,1\}$
we define $X_A=\{x\in X: A_{x_n,x_{n+1}}=1\}$. It is easy to see that $X_A$ is a closed (and therefore compact) $f$-invariant set and we say that $f|_{X_A}$ a subshift of finite type. A subshift of finite type is (topologically) mixing if $A$ is aperiodic, that is, if there exists $n\in \bN$ such that $A^n_{i,j}>0$ for all $i,j\in \cA$.

Analogously, we obtain the concept of two-sided shift spaces and shift maps by defining $X$ to be the space of all bi-infinite sequences $x=(x_n)_{n=-\infty}^\infty$ where $x_n\in \cA$ for all $n\in \bZ$.
It turns out that the shift map $f:X\to X$ (defined as in the case of one-sided shift maps) is a homeomorphism on $X$.
Analogously as in the case of one-sided shift maps we obtain sub-shifts and of sub-shifts of finite type.
 While we consider in this paper mostly the case of one-sided shift maps, all our result also hold for two-sided shift maps. For details how to make the connection between one-sided shift maps and two-sided shift maps we refer to \cite{Je}.

Let $(X_A,f)$ be a one-sided subshift of finite type.
Given $x\in X_A$ we write $\pi_n(x)=(x_1,\cdots,x_n)$. Moreover, for $\tau=(\tau_1,\cdots,\tau_n)\in \cA^n$ we denote by $\cC(\tau)=\{x\in X_A: x_1=\tau_1,\cdots, x_n=\tau_n\}$ the cylinder generated by $\tau$ and the element $\Or(\tau)=(\tau_1,...,\tau_n,\tau_1,...,\tau_n,...)$ the periodic orbit generated by $\tau$. In this case $n$ is referred to as the length of the cylinder or the orbit respectively.

Similar, we use the analogous definitions in the case of two-sided shift maps, namely $\pi_n(x)=(x_{-n},\cdots,x_n)$,  $\cC(\tau)=\{x\in X_A: x_{-n}=\tau_{-n},\cdots, x_n=\tau_n\}$ if $\tau=(\tau_{-n},\cdots, \tau_n)$ and $\Or(\tau)=(...,\tau_{-n},...,\tau_n,\tau_{-n},...,\tau_n,...)$.

It is a well-known fact that topological mixing sub-shifts of finite type have strong thermodynamic properties (see \cite{Je} and the references therein).

\subsection{Expansive homemorphisms with specification}
Let $f:X\to X$ be a homeomorphism on a compact metric space $(X,d)$. We say $f$ is expansive if there exists $\gamma>0$ (called an expansivity constant) such that if $x,y\in X$ with $d(f^n(x),f^n(y))<\gamma$ for all $n\in \bZ$ then $x=y$.

Moreover, we say $f$ has the specification property (abbreviated to "a homeomorphism with specification") if for each $\delta>0$ there exists an integer $p=p(\delta)$ such that the following holds: if

\begin{enumerate}
\item[(a)]
$I_1,\cdots, I_n$ are intervals of integers, $ I_j\subset [a,b]$ for some $a,b\in \bZ$ and all $j$.
\item[(b)] $dist(I_i,I_j)\geq p(\delta)$ for $i\not= j$,
\end{enumerate}
then for arbitrary $x_1,\cdots , x_n \in X$ there exists $x\in X$ such that
\begin{enumerate}
\item[(1)] $f^{b-a+p(\delta)}(x)=x$,
\item[(2)] $d(f^k(x),f^k(x_i))<\delta$ for $k\in I_i$.
\end{enumerate}

The specification property guarantees good mixing properties of $f$ and a sufficient number of periodic points. It turns out that topological mixing two-sided subshifts of finite type as well as diffeomorphisms
with a locally maximal topological mixing hyperbolic  set are expansive homeomorphisms with specification (see \cite{KH}).
Moreover, expansive homeomorphisms with specification  have strong thermodynamic properties, see \cite{Bo2,HRu, Ru}
We refer to  \cite{KH} for  a compact presentation.

%---------------------------------------------------------------------------
\subsection{Smooth dynamical systems}
%---------------------------------------------------------------------------

Let $M$ be a smooth Riemannian manifold and let
$f\colon M\to M$ be a $C^{1+\epsilon}$ map. We consider a compact locally
maximal $f$-invariant set $X\subset M$. Here \emph{locally
maximal} means that there exists an open neighborhood $U\subset M$
of $X$ such that $X=\bigcap_{n\in\bZ} f^n(U)$. To avoid trivialities we will always assume
that $h_{{\rm top}}(f|_X)>0$. This rules out the case that $X$ is
only a periodic orbit. Given $x\in X$ and $v\in T_xM$, we
define the \emph{Lyapunov exponent} of $v$ at $x$ (with respect to
$f$) by
\begin{equation}\label{deflya}
\lambda(x,v)\eqdef\limsup_{n\to\infty}\frac{1}{n}\log\,\lVert df^n(x)(v)\rVert
\end{equation}
with the convention that $\log 0=-\infty$.
For each $x\in X$ there exist a positive integer $s(x)\le \dim M$,
real numbers $\chi_1(x)< \cdots < \chi_{s(x)}(x)$, and linear spaces
$\{0\}=E^{0}_x\subset \cdots \subset E^{s(x)}_x=T_xM$ such that for
$i=1$, $\ldots$, $s(x)$ we have
\[
E^{i}_x=\{v\in T_xM\colon \lambda(x,v)\le \chi_i(x)\},
\]
and $\lambda(x,v)=\chi_i(x)$ whenever $v\in E^{i}_x\setminus
E^{i-1}_x$.

We will count the values of the Lyapunov exponents
$\chi_i(x)$ with their multiplicities, i.e. we consider the numbers
\[
\lambda_1(x)\le\cdots\le\lambda_{\dim M}(x),
\]
where  $\lambda_j(x)=\chi_i(x)$ for each $j\in\{\dim E^{i-1}_x+1,\cdots,\dim
E^i_x\}$.

By the Oseledets theorem, given $\mu\in \cM=\cM(f|_X)
$ the set of Lyapunov regular points (i.e. the set of points where the limit superior in~\eqref{deflya} is actually a limit) has full measure and
$\lambda_i(\cdot)$ is $\mu$-measurable. We denote by
\begin{equation}\label{deflyame}
\lambda_i(\mu)\eqdef\int\lambda_i(x) d\mu(x)
\end{equation}
the Lyapunov exponents of the measure $\mu$.
Note that if $\mu\in \cM_{\rm E}$ then
$\lambda_i(.)$ is constant $\mu$-a.e., and therefore, the corresponding value
coincides with $\lambda_i(\mu)$.
For $\mu\in \cM$
set
\[
\chi(\mu)\eqdef\min_{i=1,\ldots,\dim M} \lambda_i(\mu) =\lambda_1(\mu).
\]
Furthermore, we define the set of $\cM^+=\{\mu\in \cM\colon \chi(\mu)>0\}$ and
$\cM^+
_{\rm E}= \cM^+
\cap \cM_{\rm E}$.

 For $x\in \Per_n(f)$ we have that $\lambda_i(x)=\frac{1}{n}\log
\lvert \delta_i\rvert$, where $\delta_i$ are the eigenvalues of $Df^n(x)$.
We say a periodic point $x$ is \emph{expanding} if $\lambda_1(x)>0$. Let $\EPer_n(f)$ denote the fixed points of $f^n$ which are expanding. Hence,
$\EPer(f)=\bigcup_{n\in \bN} \EPer_n(f)$ is the set of all expanding periodic points.

Given  $\alpha>0$ and $c>0$ we define
\begin{equation}\label{ne}
X_{\alpha,c}=\{x\in M\colon\lVert Df^k(x)\rVert_{\rm co} \geq ce^{k\alpha} \text{ and all }k\in\bN\}.
\end{equation}
Here  $\lVert Df^k(x)  \rVert_{\rm co}$ is the minimums norm (also called co-norm) which coincides with $\lVert (Df^k(x))^{-1}\rVert^{-1}$ if $f$ is a local diffeomorphism at $x$.
Moreover, we say that a compact forward invariant set $X\subset M$ is  \emph{uniformly expanding} if there exist constants  $c>0$ and $\alpha>0$ such that $X\subset X_{\alpha, c}$.
For convenience we sometimes also refer to relative compact forward invariant sets contained in some $X_{\alpha,c}$ as uniformly expanding sets.
We denote by $\chi(X)$ the largest $\alpha>0$ such that $X\subset X_{\alpha,c}$ for some $c>0$.
It is easy to see that if $x\in \EPer_n(f)$ then there exists  $c=c(x)>0$ such that for all
integers $k\ge 0$ and $0\le i\le n-1$ we have
\begin{equation}\label{eqhi}
ce^{k\lambda_{1}(x)} \leq \lVert (Df^k(f^i(x)))^{-1}  \rVert^{-1}.
\end{equation}
For $\alpha>0, c>$ and $n\in\bN$ we define
\begin{multline}\label{eqexpand}
%\begin{split}
\EPer_n(\alpha, c) = \{x\in \Per_n(f)\colon \lVert(Df^{-k}(f^i(x)))^{-1}\rVert^{-1}
\ge c e^{k\alpha}\colon\\
 \text{ for all } k\in \bN \text{ and } 0\le i\le n-1 \}.
%\colon
%\lVert(Df^{k}(f^i(x))\rVert\ge c e^{k\alpha}\\
% \text{ for all } k\ge 1 \text{ and all } 0\le i < n \}.
%\end{split}
\end{multline}
It follows that, if  $\alpha\geq \alpha'$, $c\geq c'$, then
\begin{equation}\label{ni}
\EPer_n(\alpha,c) \subset \EPer_n(\alpha',c')
\end{equation}
and
\begin{equation}
\EPer(f)=\bigcup_{\alpha>0}\bigcup_{c>0}\bigcup_{n=1}^\infty
\EPer_n(\alpha,c).
\end{equation}
Given $\Phi\in C(X,\bR^m)$, $w\in \R(\Phi), n\in \bN$ and $r>0$ we define
\begin{equation}\label{eqsirot}
  \EPer_n(w,r,\alpha,c) = \EPer_n(\alpha, c)\cap \Per_n(w,r).
\end{equation}

We will need the following construction of uniformly expanding sets.
Let $\alpha>0$, $c>0$, $w\in \R(\Phi)$ and $r>0$ such that $\EPer_n(w,r,\alpha,c)\not=\emptyset$ for some $n\in \bN$. We define
\begin{equation}\label{wagner}
X= X_{w,r,\alpha,c}
= \overline{\bigcup_{n=1}^\infty \EPer_n(w,r,\alpha, c)} .
\end{equation}
A simple continuity argument shows that $X$ is a uniformly expanding set.
Moreover, for every $\mu\in \cM(f|_{X})$ we have $\rv(\mu)\in D(w,r)$.
Furthermore, for every $n\in\bN$  we have
\begin{equation}\label{eqsi}
  \Per_n(f)\cap X = \EPer_n(w,r,\alpha,c).
\end{equation}

We will need the following classical result (see for example~\cite{Ru}).

\begin{proposition}\label{ha}
 Let $f\colon M\to M$ be a $C^{1+\epsilon}$ map and let $X\subset M$ be
 a compact uniformly expanding  set of $f$. Then we have,
 \begin{equation}\label{eqdruck}
 \limsup_{n\to\infty}\frac{1}{n}\log {\rm card}\left( \Per_n(f)\cap X \right) \le h_{\rm top}(f|_X).
 \end{equation}
 Furthermore, if $f|_X$ is topologically mixing then we have
 equality in~\eqref{eqdruck}, and the limsup is in fact a limit.
\end{proposition}

\section{The geometry of rotation sets}
In this section we study the question if every compact and convex set
is attained as a rotation set within the class of subshifts of finite type.

It follows from the facts that $\cM$ is a compact  convex set, and that $\mu\mapsto \rv(\mu)$ is continuous and affine that the rotation set is always compact and convex. To show that for shift maps,
these are  the sole restrictions, we explicitly construct for an arbitrary convex compact set $K$ in $\bR^m$  a continuous potential on a shift-space which generates $K$ as its rotation set.

Throughout this section we will use the notation from Section 2.3. The following theorem considers the 2-dimensional case. Subsequently, we will prove the corresponding statement in the $m$-dimensional setting.

\begin{theorem} \label{thm1}  Let $K$ be a compact convex subset of $\mathbb{R}^2$ and $( X, f)$ be a full one-sided shift with alphabet $\cA=\{0,1\}$. Then there exists a continuous potential $\Phi: X\to \bR^2$ such that $\R(\Phi)=K$.
\end{theorem}
\begin{proof}
In case of empty interior, $K$ is reduced to either a single point or a line segment. In those situations the construction of $\Phi$ is trivial. Hence, we only need to consider the case when $K$ has non-empty interior. Note also that since $K$ is a compact convex set in $\bR^2$, the boundary of $K$ is Lipschitz and is therefore rectifiable. After normalizing we may assume the boundary has length one.

The proof is based on an approximation argument. We create a sequence of continuous potentials $\Phi_n$ which converges uniformly to a potential $\Phi$. At the same time, we assure that $\R(\Phi_n)$ converge to $K$. This is done by approximating $K$ by polygons with vertices on the boundary of $K$. Ideally, the rotation sets of $\Phi_n$ are equal to these polygons, but this will be  not  possible to be achieved without compromising the uniform convergence. It is sufficient however, for $\R(\Phi_n)$ to be suitably close to those polygons.

We will construct a sequence $\{\Phi_n\}_{n=0}^\infty$ of potentials on $ X$ with the following properties:
\begin{enumerate}
        \item \label{assumptions1} $\Phi_n: X\to \bR^2$ is continuous
        \item \label{assumptions2} $\|\Phi_{n+1}-\Phi_n\|_\infty\le\frac{11}{8}\cdot\frac{1}{2^{n}}$
        \item \label{assumptions3} There exists a set of $2^{n+1}$ equidistant points $\{w_{n,j}\}_{j=1}^{2^{n+1}}$ on the boundary of $K$ and a corresponding set $\{w^*_{n,j}\}_{j=1}^{2^{n+1}}$ of points in $K$ such that
            $${\text{conv}}\{w^*_{n,j}\}\subset \R(\Phi_n)\subset K\quad\text{and}\quad \|w_{n,j}-w^*_{n,j}\|_2\le\frac{5}{4}\cdot\frac{1}{6^{n}},$$
where ${\rm conv}\ A$ denotes the convex hull of the set $A$.

\end{enumerate}

The sequence will be explicitly constructed by induction on n. We will hence start with the case n=0.
To initiate, we place two equidistant points $w_{0,1}$ and $w_{0,2}$ on the boundary of $K$. We split $ X$ into 2 cylinders of length one generated by (0) and (1)  and define
$$\Phi_0(x)=\begin{cases} w_{0,1}\quad if\,\, x\in\C(0)\\ w_{0,2}\quad if\,\, x\in\C(1)\end{cases}$$
Then $\R(\Phi_0)$ is exactly the line segment connecting points $w_{0,1}$ and $w_{0,2}$. We call $\C_{0,1}=\C(0)$ and $\C_{0,2}=\C(1)$ the original cylinders of step 0, which correspond to the points $w_{0,1}$ and $w_{0,2}$.

To illustrate the process, we will also show how to move from $n=0$ to $n=1$ here (see Figure \ref{RotSet}). Afterwards, we will demonstrate the general step from $n$ to $n+1$.

We move along the boundary of $K$ by $\frac18$ ($\frac14$ of the distance between $w_{0,1}$ and $w_{0,2}$) to the left and to the right of $w_{0,1}$ and place points $L_{0,1}$ and $R_{0,1}$ respectively. Similarly, we place points $L_{0,2}$ and $R_{0,2}$ around $w_{0,2}$. Then $L_{0,1}, R_{0,1}, L_{0,2}, R_{0,2}$ are four equidistant points on the boundary of $K$. This is shown on the left hand side of Figure \ref{RotSet}. To pass to the next potential $\Phi_1$ we start by taking $\Phi_1=\Phi_0$ and then modify it as follows.

We build the original cylinders of step 1 by
\begin{enumerate}
   \item [(a)] replicating the generators of the original cylinders of the previous step three times: $\C(0)\leadsto\C(000),\,\C(1)\leadsto\C(111)$
   \item [(b)] altering the last element to be either 0 or 1: $\C(000)\leadsto\{\C(000),\C(001)\}$ and $\C(111)\leadsto\{\C(111),\C(110)\}$
 \end{enumerate}
 We name the four cylinders obtained: $\C_{1,1}=\C(000), \C_{1,2}=\C(001), \C_{1,3}=\C(111), \text{ and }\C_{1,4}=\C(110)$. We assign $\Phi_1$ the value of the point to the left of the value of $\Phi_0$, if the last element was not changed, and the point to the right, if the last element was changed, e.g. $\Phi_1=R_{0,1}$ on $\C(001)$. As a result, the rotation set of $\Phi_1$ is contained in the polygon spanned by all the points (see left hand side of Figure \ref{RotSet}). However we need to make sure that $\R(\Phi_1)$ is sufficiently close to this polygon. To get rotation vectors which are close to the vertices, it is necessary to change the values of the potential on certain periodic orbits. A subtlety about doing this is that we should not change the values of $\Phi_1$ on the whole orbit, since we need these values to remain close to $\Phi_0$.

To be precise, each of the original cylinders contains a periodic orbit with the same generator. We define $\Phi_1$ to be the same on all cylinders generated by the first 3 elements of the shifted periodic orbits of such form. For example,  since $\C(001)$ supports $\Or(001)$ and $ f\Or(001)=\Or(010)$, we set $\Phi_1=R_{0,1}$ on $\C(010)$. We keep $\Phi_1$ at all the other points the same as $\Phi_0$. Summarizing all of the above we have
\begin{equation}\label{phi_1} \Phi_1(x)=\begin{cases} L_{0,1},\quad if\,\, x\in\C(\pi_3\circ f^k\Or(000)),\quad k=0,1\\
                          R_{0,1},\quad if\,\, x\in\C(\pi_3\circ f^k\Or(001)),\quad k=0,1\\
                          L_{0,2},\quad if\,\, x\in\C(\pi_3\circ f^k\Or(111)),\quad k=0,1\\
                          R_{0,2},\quad if\,\, x\in\C(\pi_3\circ f^k\Or(110)),\quad k=0,1\\
                          \Phi_0(x),\quad \text{otherwise}
            \end{cases}
\end{equation}

The potential $\Phi_1$ is constant on cylinders of length $m_1=3$, and the value of $\Phi_1$ on each such cylinder is a point on the boundary of $K$.  Since these cylinders form a disjoint partition of $ X$ by clopen sets, $\Phi_1$ is continuous with respect to the product topology.

Note that $\Phi_1(x)$ and $\Phi_0(x)$ may have different values only when $x$ is in one of the cylinders listed above. In this case $\|\Phi_0(x)-\Phi_1(x)\|$ is either $\|L_{0,i}-w_{0,i}\|$ or $\|R_{0,i}-w_{0,i}\|$ for $i=1,2$, and therefore cannot exceed $\frac18$. Hence, $$\|\Phi_1-\Phi_0\|_\infty=\sup_{x\in X}\|\Phi_1(x)-\Phi_0(x)\|\le\frac{1}{8}<\frac{11}{8}\cdot\frac{1}{2^0}$$

For $j=1,...,4$ let $w_{1,j}=\Phi_1(\C_{1,j})$ denote the point on the boundary of $K$, $\Or_{1,j}$ be the orbit with the same generator as $\C_{1,j}$ and $w^*_{1,j}=\rv_{\Phi_1}(\mu_{1,j})$, where $\mu_{1,j}$ is a $ f$-invariant measure supported on $\Or_{1,j}$ (see the right hand side of Figure \ref{RotSet}). Note that from here on we use an additional index for rotation vectors to emphasize the underlying potential. Then $w^*_{1,1}=w_{1,1},\, w^*_{1,3}=w_{1,3}$. However,
\begin{equation}
\begin{split}
w^*_{1,2}&=\mu_{1,2}(\Or(001))\Phi_1(\C(001))\\
&\quad +\mu_{1,2}(\Or(010))\Phi_1(\C(010))+\mu_{1,2}(\Or(100))\Phi_1(\C(100))\\
&=\frac13 w_{1,2}+\frac13 w_{1,2}+\frac13 w_{0,2}\\
&=\frac23 w_{1,2}+\frac13 w_{0,2}
\end{split}
\end{equation}
and $w^*_{1,4}=\frac23 w_{1,4}+\frac13 w_{0,1}$.

We compute the distance between corresponding points and obtain
$$\|w_{1,2}-w^*_{1,2}\|_2=\|w_{1,2}-\frac23 w_{1,2}-\frac13 w_{0,2}\|_2 =\frac13\|w_{1,2}-w_{0,2}\|_2\le\frac{1}{3}\cdot\frac38\le\frac54\cdot\frac16$$
and $\|w_{1,4}-w^*_{1,4}\|_2\le\frac54\cdot\frac16$.

For any measure $\mu\in\cM$, the point
\begin{equation}
\rv_{\Phi_1}(\mu)=\sum_{C\text{-cylinder of length 3}}\mu(C)\Phi_1(C)
\end{equation}
 is a convex combination of the points $\{w_{i,j}: i=0,1;j=1,...,2^{i+1}\}$ in $K$. Therefore, $\R(\Phi_1)$ is a subset of $K$. On the other hand, $\R(\Phi_1)$ contains points $w^*_{1,j}$ and hence ${\text{conv}}\{w^*_{1,j}\}_{j=1}^4\subset \R(\Phi_1)$.
\begin{figure}[h]
\begin{center}
\psfragscanon
\psfrag{1}[c][l]{{\huge{${\R (\Phi_{1})}$}}}
\psfrag{2}[c][l]{{\huge{is inside}}}
\psfrag{9}[c][l]{{\huge{inside}}}
\psfrag{3}[c][l]{{\huge{$w_{0,1}$}}}
\psfrag{4}[c][l]{{\huge{$L_{0,1}$}}}
\psfrag{5}[c][l]{{\huge{$R_{0,2}$}}}
\psfrag{6}[c][l]{{\huge{$w_{0,2}$}}}
\psfrag{7}[c][l]{{\huge{$L_{0,2}$}}}
\psfrag{8}[l][l]{{\huge{$R_{0,1}$}}}
\psfrag{5*}[c][l]{{\huge{$w_{1,4}^*$}}}
\psfrag{8*}[c][l]{{\huge{$w_{1,2}^*$}}}
\psfrag{4'}[c][l]{{\huge{$w_{1,1}$}}}
\psfrag{5'}[c][l]{{\huge{$w_{1,4}$}}}
\psfrag{7'}[c][l]{{\huge{$w_{1,3}$}}}
\psfrag{8'}[l][l]{{\huge{$w_{1,2}$}}}

\scalebox{0.45}{\includegraphics{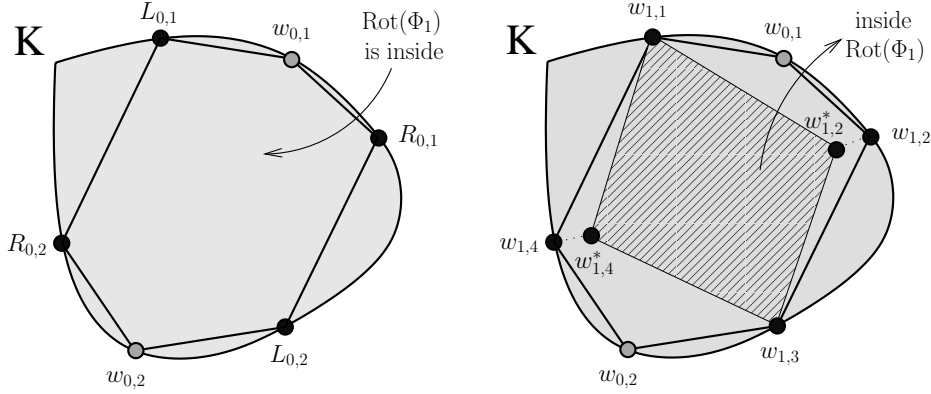}}
\caption{This figure illustrates step from $n=0$ to $n=1$ in the proof of Theorem \ref{thm1}}
\label{RotSet}
\end{center}
\end{figure}

Now we are ready to follow the general induction step from $n$ to $n+1$.
Suppose we have defined the sequence of potentials $\Phi_0,\Phi_1,...,\Phi_{n}$ satisfying (\ref{assumptions1}),(\ref{assumptions2}),(\ref{assumptions3}). Here $\Phi_{n}$ is constant on cylinders of size $m_{n}$. We introduce a new set of points on the boundary of $K$. For each $j$ denote by $L_{n,j}$ the point located exactly $\frac14\cdot\frac{1}{2^{n}}$ along the boundary to the left of $w_{n,j}$ and by $R_{n,j}$ the point of the same distance to the right. Then the set $\{L_{n,j},R_{n,j}\}_{j=1}^{2^{n+1}}$ gives $2^{n+2}$ equidistant points on the boundary of $K$. We enumerate these points starting from $L_{n,1}$ and moving to the right along the boundary: $w_{n+1,1}=L_{n,1},\, w_{n+1,2}=R_{n,1},\, w_{n+1,3}=L_{n,2},\, w_{n+1,4}=R_{n,2}$ and so on. The general formula is
\begin{equation}
w_{n+1,i}=\left\{
    \begin{array}{ll}
      $$L_{n,j}$$, & \hbox{if $i$ is odd for $j=\frac12(i+1)$;} \\
      $$R_{n,j}$$, & \hbox{if $i$ is even for $j=\frac12i$.}
    \end{array}
  \right.
\end{equation}
We denote the original cylinders of step $n$ by $\{\C_{n,j}\}_{j=1}^{2^{n+1}}$. Then for $j=1,...,2^{n+1}$ the corresponding point on the boundary of $K$ $w_{n,j}=\Phi_{n}(\C_{n,j})$. Furthermore, let $\Or_{n,j}$ be the orbit with the same generator as $\C_{n,j}$ and  $w^*_{n,j}=\rv_{\Phi_{n}}(\mu_{n,j})$, where $\mu_{n,j}$ is a $ f$-invariant measure supported on $\Or_{n,j}$.

We construct the original cylinders of step $n+1$ by replicating the generator of each original cylinder of the previous step $m_{n+1}=3^{n+1} m_{n}$ times and then changing the last entry to be either 0 or 1. E.g. for $j=1,...,2^{n+1}$ define
\begin{equation}
\tau_{j}=(\pi_{m_{n+1}}\Or_{n,j})\quad\text{and}\quad\bar{\tau}_{j}=(\pi_{m_{n+1}-1}\Or_{n,j},a)
\end{equation}
 where $a$ is a binary complement to the last entry of $\tau_j$. Then the sequence of $m_{n+1}$-tuples $\tau_j$ and $\bar{\tau}_j$  will be the generators of the original cylinders of step $n+1$. For each cylinder of step $n$ we obtain 2 new original cylinders of step $n+1$. Let $k_n=m_{n+1}-m_n=(3^n-1)m_n$. We define $\Phi_{n+1}$ as follows
$$\Phi_{n+1}(x)=\begin{cases}
L_{n,j},&\text{if}\, x\in\C(\pi_{m_{n+1}}\circ f^k\Or(\tau_{j})),\, k=0,...,m_{n+1}-1\\
R_{n,j},&\text{if}\, x\in\C(\pi_{m_{n+1}}\circ f^k\Or(\bar{\tau}_{j})),\, k=0,...,k_n-1\\
\Phi_n( f^{k_n}\Or(\bar{\tau}_{j})),&\text{if}\, x\in\C(\pi_{m_{n+1}}\circ f^k\Or(\bar{\tau}_{j})),\, k=k_n,...,m_{n+1}-1\\
\Phi_n(x),&\text{otherwise}
                \end{cases}$$

By definition, $\Phi_{n+1}$ is constant on cylinders of length $m_{n+1}$ and therefore it is continuous on $ X$.

Note that $ f^{k_n}\Or(\bar{\tau}_{j})$ and $\Or_{n,j}$ have the same first $m_n-1$ entrees and different $m_n$-th entree. Therefore, these elements are mapped into consecutive points under $\Phi_n$. Since $\Phi_n(\Or_{n,j})=w_{n,j}$, either $\Phi_n( f^{k_n}\Or(\bar{\tau}_{j}))=w_{n,j-1}$ or $\Phi_n( f^{k_n}\Or(\bar{\tau}_{j}))=w_{n,j+1}$. (Here additive operations on j are performed modulo $2^{n+1}$).

Now we will estimate $\sup_{x\in F}\|\Phi_{n+1}(x)-\Phi_n(x)\|_2$. First, suppose $x\in\C(\pi_{m_{n+1}}\circ f^k\Or(\tau_{j}))$ for some $k$ between 0 and $m_{n+1}-1$ and $j$ between 1 and $2^{n+1}$. Then $\Or(\tau_j)=\Or_{n,j}$ and $\Phi_n(\Or_{n,j})=w_{n,j}$.

If $j$ is odd, then $w_{n,j}=L_{n-1,i}$ for $i=\frac12(j+1)$ and hence
\begin{equation}
\Phi_n(x)=\Phi_n( f^k\Or_{n,j})=L_{n-1,i}=w_{n,j}\quad\text{for all k}.
\end{equation}
If $j$ is even, then  $w_{n,j}=R_{n-1,i}$ for $i=\frac12j$. Therefore $\Phi_n(x)=\Phi_n( f^k\Or_{n,j})=R_{n-1,i}$ for $k<k_{n-1}$ (here $k_{n-1}=m_n-m_{n-1}$) and $\Phi_n(x)=\Phi_{n-1}( f^{k_{n-1}}\Or_{n,j})$ for $k_{n-1}\le k<m_n$. Then the same values repeat periodically for $k$ on each of the intervals of length $m_n$, that is,
$$\Phi_n(x)=\begin{cases}R_{n-1,i}, &\text{ for }sm_n\le k<sm_m+k_{n-1};\\
\Phi_{n-1}( f^{k_{n-1}}\Or_{n,j}), &\text{ for } sm_n+k_{n-1}\le k<(s+1)m_n,\\
\end{cases}$$
where $s=0,...,3^{n+1}-1$.
In this case either $\Phi_{n-1}( f^{k_{n-1}}\Or_{n,j})=w_{n-1,i-1}$ or $\Phi_{n-1}( f^{k_{n-1}}\Or_{n,j})=w_{n-1,i+1}$. Summarizing the above we get
$$\|\Phi_{n+1}(x)-\Phi_n(x)\|_2\le \max\{\|L_{n,j}-w_{n,j}\|,\,\|L_{n,j}-w_{n-1,i-1}\|,\,\|L_{n,j}-w_{n-1,i+1}\|\}$$
Since the distance along the boundary between $L_{n,j}$ and $w_{n,j}$ is $\frac14\cdot\frac{1}{2^{n+1}}$, \\ $\|L_{n,j}-w_{n,j}\|\le \frac{1}{2^{n+3}}$. Also,
\begin{align*}\|L_{n,j}-w_{n-1,i\pm 1}\| &\le\|L_{n,j}-w_{n,j}\|+\|w_{n,j}-w_{n-1,i\pm 1}\|\\
&=\|L_{n,j}-w_{n,j}\|+\|R_{n-1,i}-w_{n-1,i\pm 1}\|\\
&\le \|L_{n,j}-w_{n,j}\|+\|R_{n-1,i}-w_{n-1,i}\|+\|w_{n-1,i}-w_{n-1,i\pm1}\|\\
&\le \frac14\cdot\frac{1}{2^{n+1}}+\frac14\cdot\frac{1}{2^{n}}+\frac{1}{2^{n}}\\
&\le \frac{11}{8}\cdot\frac{1}{2^{n}}.
\end{align*}
The estimates are similar when $x\in\C(\pi_{m_{n+1}}\circ f^k\Or(\bar{\tau}_{j})),\,\,\, k=0,...,k_n-1$. In this case, since $\tau_j$ and $\bar{\tau}_j$ differ only by $m_{n+1}$ entry, $\pi_{m_n}\circ f^k\Or(\bar{\tau}_j)=\pi_{m_n}\circ f^k\Or_{n,j}$. Thus, $\Phi_n(x)=\Phi_n( f^k\Or_{n,j})$. By replacing $L_{n,j}$ by $R_{n,j}$ in the arguments above we obtain
$$\|\Phi_{n+1}(x)-\Phi_n(x)\|_2\le\max\{\|R_{n,j}-w_{n,j}\|_2,\,\|R_{n,j}-w_{n-1,i\pm 1}\|_2\}\le\frac{11}{8}\cdot\frac{1}{2^n}.$$

Finally, consider $x\in\C(\pi_{m_{n+1}}\circ f^k\Or(\bar{\tau}_{j}))$ when $k=k_n,...,m_{n+1}-1$. As discussed before, either $\Phi_{n+1}(x)=w_{n,j-1}$ or $\Phi_{n+1}(x)=w_{n,j+1}$. We assume $\Phi_{n+1}(x)=w_{n,j-1}$, the other case is similar. If $j$ is even, $j-1$ is odd and $w_{n,j-1}=L_{n-1,i}$ for $i=\frac12j$. Thus $\Phi_n(x)=\Phi_n( f^k\Or(\bar{\tau}_j))=L_{n-1,i}=\Phi_{n+1}(x)$ for all $k=k_n,...,m_{n+1}-1$. If $j$ is odd, $j-1$ is even and $w_{n,j-1}=R_{n-1,i}$ for $i=\frac12(j-1)$. Then $\Phi_n( f^k\Or(\bar{\tau}_j))=R_{n-1,i}$ for $k_n\le k<k_n+k_{n-1}$ and $\Phi_n( f^k\Or(\bar{\tau}_j))=\Phi_{n-1}( f^k_{n-1}\Or(\bar{\tau}_j))=w_{n-1,i\pm1}$ for $k_n+k_{n-1}\le k<m_{n+1}$. Since
\begin{align*}\|R_{n-1,i}-w_{n-1,i\pm1}\|_2&\le\|R_{n-1,i}-w_{n-1,i}\|_2+\|w_{n-1,i}-w_{n-1,i\pm1}\|_2\\ &\le\frac14\cdot\frac{1}{2^n}+\frac{1}{2^n}\\
&\le\frac54\cdot\frac{1}{2^n}
\end{align*}
we conclude that
\begin{equation}
\|\Phi_{n+1}(x)-\Phi_n(x)\|_2\le\frac54\cdot\frac{1}{2^n}.
\end{equation}
It follows that
\begin{equation}
\sup_{x\in F}\|\Phi_{n+1}(x)-\Phi_n(x)\|_2\le\frac{11}{8}\cdot\frac{1}{2^n},
\end{equation}
 which concludes the proof of (\ref{assumptions2}). Next, we prove part (\ref{assumptions3}).

On any cylinder of size $m_{n+1}$ the potential $\Phi_{n+1}$ is equal to one of the points on the boundary of $K$ that were obtained during this or one of the previous steps. Since these cylinders form a disjoint partition of $ X$, for any probability measure $\mu$ the rotation vector $\rv_{\Phi_{n+1}}(\mu)$ is a convex combination of the boundary points of $K$. Convexity of $K$ implies that $\R(\Phi_{n+1})\subset K$.

As before, we denote the original cylinders of step $n+1$ by $\{\C_{n+1,j}\}_{j=1}^{2^{n+2}}$. Then for $j=1,...,2^{n+2}$ let $w_{n+1,j}=\Phi_{n+1}(\C_{n+1,j})$ be the corresponding point on the boundary of $K$, $\Or_{n+1,j}$ be the orbit with the same generator as $\C_{n+1,j}$ and $w^*_{n+1,j}=\rv_{\Phi_{n+1}}(\mu_{n+1,j})$, where $\mu_{n+1,j}$ is a $ f$-invariant measure supported on $\Or_{n+1,j}$. Since $w^*_{n+1,j}$ are rotation vectors of $\Phi_{n+1}$ and the rotation set is convex, ${\text{conv}}\{w^*_{n+1,j}\}_j\subset \R(\Phi_{n+1})$.

Using the fact that $\mu_{n+1,j}( f^k\Or_{n+1,j})=\frac{1}{m_{n+1}}$ for all $k$, we compute
\begin{equation}
w^*_{n+1,j}=\rv_{\Phi_{n+1}}(\mu_{n+1,j})=\frac{1}{m_{n+1}}\sum_{k=0}^{m_{n+1}-1}\Phi_{n+1}( f^k\Or_{n+1,j}).
\end{equation}
If $j$ is odd, $\Phi_{n+1}(\Or_{n+1,j})=w_{n+1,j}=L_{n,i}$  for $i=\frac12(j+1)$. Then $\Phi_{n+1}(\Or^k_{n+1,j})=L_{n,i}=w_{n+1,j}$ for all $k$ and hence $w^*_{n+1,j}=w_{n+1,j}$. If $j$ is even, $\Phi_{n+1}(\Or_{n+1,j})=w_{n+1,j}=R_{n,i}$ for $i=\frac12j$. Then $\Phi_{n+1}(\Or^k_{n+1,j})=R_{n,i}=w_{n+1,j}$ for $k<k_n$ and $\Phi_{n+1}(\Or_{n+1,j})=w_{n+1,j}=L_{n,i}$ for $i=\frac12(j+1)$. Then $\Phi_{n+1}(\Or^k_{n+1,j})$ is either $w_{n,i-1}$ or $w_{n,i+1}$ for $k\ge k_n$. Assuming $\Phi_{n+1}(\Or^k_{n+1,j})=w_{n,i-1}$ we get
\begin{equation}
w^*_{n+1,j}=\frac{1}{m_{n+1}}\left(\sum_{k=0}^{k_{n}-1}w_{n+1,j}+\sum_{k=k_n}^{m_{n+1}-1}w_{n,i-1}\right).
\end{equation}
Hence,
\begin{align*}
\|w_{n+1,j}-w^*_{n+1,j}\|_2&=\frac{1}{m_{n+1}}\|\sum_{k=k_n}^{m_{n+1}-1}w_{n+1,j}-w_{n,i-1}\|_2\\
&\le\frac{m_{n+1}-k_n}{m_{n+1}}\|w_{n+1,j}-w_{n,i-1}\|_2\\
&\le\frac{m_{n}}{m_{n+1}}(\|R_{n,i}-w_{n,i}\|_2+\|w_{n,i}-w_{n,i-1}\|_2)\\
&\le\frac{1}{3^{n+1}}\left(\frac14\cdot\frac{1}{2^n}+\frac{1}{2^n}\right)\\
&=\frac{5}{4}\cdot\frac{1}{6^{n+1}}.
\end{align*}
This proves (\ref{assumptions3}), and the induction is complete.\\

It follows from (\ref{assumptions1}) and (\ref{assumptions2}) that the sequence $(\Phi_n)_{n\in\bN}$ converges uniformly to a continuous potential $\Phi$. For any $\epsilon>0$ there is $N$ such that $\|\Phi-\Phi_n\|_\infty<\epsilon$ whenever $n>N$. Then for $x\in \R(\Phi)$ and $x=\rv_{\Phi}(\mu_x)$ we have
\begin{equation}
\begin{split}
d(x,\R(\Phi_n))&=\inf_{y\in \R(\Phi_n)}\|x-y\| \\
&\le \|\rv_{\Phi}(\mu_x)-\rv_{\Phi_n}(\mu_x)\|\le\int\|\Phi-\Phi_n\|d \mu_x<\epsilon.
\end{split}
\end{equation}
Moreover, for $y\in \R(\Phi_n)$ and  $y=\rv_{\Phi_n}(\mu_y)$ we have
\begin{equation}
\begin{split}
d(y,\R(\Phi))&=\inf_{x\in \R(\Phi)}\|x-y\|\le\|\rv_{\Phi}(\mu_y)-\rv_{\Phi_n}(\mu_y)\| \\
&\le\int\|\Phi-\Phi_n\|d \mu_y<\epsilon.
\end{split}
\end{equation}
Therefore, the Hausdorff distance $d_H(\R(\Phi),\R(\Phi_n))<\epsilon$ for all $n>N$ and hence $\R(\Phi_n)$ converges to $\R(\Phi)$.

On the other hand, ${\text{conv}}\{w^*_{n,j}\}_j\subset \R(\Phi_n)\subset K$ and the polygons ${\text{conv}}\{w^*_{n,j}\}_j$ converge to $K$ with respect to the Hausdorff metric. Thus, $\R(\Phi_n)$ also converges to $K$ as $n\to\infty$. We obtain $\R(\Phi)=K$, which concludes the proof of the theorem.

\end{proof}

\begin{remarks}{\rm (i) }
The potential obtained in  Theorem \ref{thm1} is not H\"older continuous. To see this, consider $x=(0000000...)$ and $x_n=\Or_{n,2}$ which is a periodic point generated by a $m_n$-tuple $(000...01)$. Then $\Phi(x)$ is the point located $\frac14$ along $\partial K$ to the left of $w_{0,1}$. Point $\Phi(x_n)$ is found by starting at $w_{0,1}$, then moving $\frac14$ to the left and $\frac12\cdot\frac{1}{2^n}$ to the right along  $\partial K$. Then the distance between $\Phi(x)$ and $\Phi(x_n)$ decreases as $2^{-n}$, however $d(x,x_n)=2^{-m_n}$ where $m_n\approx 3^{n^2}$.\\
{\rm (ii)}
While Theorem \ref{thm1} is formulated for a one-sided full shift, the procedure of the proof can be easily generalized to (one and two-sided) shifts and topologically mixing subshifts of finite type.
In case of a topologically mixing subshift $(X_A,f)$ of finite type the fact that $f$ has positive topological entropy guarantees that $f$ has sufficiently many periodical points that can be used to construct the
sequence of potentials $(\Phi_n)_{n\in \bN}$ converging to a potential $\Phi$ and satisfying similar conditions as {\rm (1), (2), (3)}.
\end{remarks}

We now  generalize Theorem \ref{thm1} to $\bR^m$.

\begin{theorem} \label{thm2m} Let $K$ be a compact convex subset of $\mathbb{R}^m$. Then there exist a full one-sided shift map  $(X,f)$ with  finite alphabet and a continuous potential $\Phi: X\to \bR^m$ such that $\R(\Phi)=K$.
\end{theorem}
\begin{proof} We can assume that $K$ has non-empty interior, since otherwise we can repeat the argument of the proof in a lower dimension.
First note that any open bounded convex set in $\bR^m$ has Lipschitz boundary (\cite[Corollary 1.2.2.3]{grisvard}). Then there is a finite covering of the boundary of $K$ by open balls $U_1,...,U_N$ with centers at points $A_1,...,A_N\in\partial K$ of radii $r_1,...,r_N$ and bijective maps $h_1,...,h_N$ from each neighborhood into the unit ball $B$ of $\bR^m$ such that \begin{enumerate}
                                        \item [(1)] $h_i$ and $h_i^{-1}$ are Lipschitz continuous with constant $L$.
                                        \item [(2)] $h_i(\partial K\cap U_i)=\{(y_1,...,y_m)\in B:y_m=0\}=B_0$
                                        \item [(3)] $h_i({\rm int}\ K\cap U_i)=\{(y_1,...,y_m)\in B:y_m>0\}=B_+$
                                      \end{enumerate}
For convenience we endow $\bR^m$ with the maximums norm. Let $ X$ be a one-sided shift space with alphabet $\cA=\{0, 1, 2, 3,..., 2N-2, 2N-1\}$. For each coordinate $y_1,...,y_{m-1}$, we apply the procedure used in the proof of the previous theorem to the interval $[-1,1]$ instead of the boundary curve. Then the functions $h_1^{-1},...,h_N^{-1}$ map each point in $B_0$ into $N$ points on the boundary of $K$. To compensate, the potentials $\Phi_n$ will map the cylinders whose generators have only entries from a subset of the alphabet $\{i-1, i\}$ into the points in $U_i$.

Suppose $i$ is any integer between 1 and $N$. We place two initial points in $B_0$ with coordinates $w_{0,1}=(-\frac12,0,0,...,0)$ and $w_{0,2}=(\frac12,0,0,...,0)$. The original cylinders of step zero will be $\C_{0,1}^{(i)}=\C(i-1)$ and $\C_{0,2}^{(i)}=\C(i)$ and we set $\Phi_0(\C_{0,1}^{(i)})=h_i^{-1}(w_{0,1})$,  $\Phi_0(\C_{0,2}^{(i)})=h_i^{-1}(w_{0,2})$. The original cylinders of step one $\C_{1,j}^{(i)}, \quad (j=1,2,3,4)$ are constructed in the same way as above using elements $i-1$ and $i$ instead of 0 and 1. To select the points in $B_0$ for this step we use the second coordinate $y_2$ and move to the "left" (negative direction of the $y_2$-axis) and to the "right" (positive direction) by $\frac12$. We obtain the following set of points in $B_0$:
$w_{1,1}=L_{0,1}=(-\frac12,-\frac12,0,...,0)$, $w_{1,2}=R_{0,1}=(-\frac12,\frac12,0,...,0)$, $w_{1,3}=L_{0,2}=(\frac12,-\frac12,0,...,0)$, $w_{1,4}=R_{0,2}=(\frac12,\frac12,0,...,0)$. Potential $\Phi_1$ is defined as in (\ref{phi_1}):
$$\Phi_1(x)=\begin{cases} h_i^{-1}(L_{0,1}),\quad if\,\, x\in\C(\pi_3\circ f^k\Or(i-1,i-1,i-1)),\quad k=0,1\\
                          h_i^{-1}(R_{0,1}),\quad if\,\, x\in\C(\pi_3\circ f^k\Or(i-1,i-1,i)),\quad k=0,1\\
                          h_i^{-1}(L_{0,2}),\quad if\,\, x\in\C(\pi_3\circ f^k\Or(i,i,i)),\quad k=0,1\\
                          h_i^{-1}(R_{0,2}),\quad if\,\, x\in\C(\pi_3\circ f^k\Or(i,i,i-1)),\quad k=0,1\\
                          \Phi_0(x),\quad \text{otherwise}
            \end{cases}
$$
Then we repeat the process with the next coordinate $y_3$ and use the original cylinders of step two to define $\Phi_3$. After we are finished with coordinate $m-1$ we start again with $y_1$, but now we shift to the left and right by an additional $\frac14$. Since the number of points in $B_0$ will double each time, we can use the original cylinders of the next step to define the next potential in the sequence. As in the proof of the previous theorem, we obtain a sequence of continuous potentials $\Phi_n$ such that
\begin{equation}
\|\Phi_{2(m-1)(n+1)}-\Phi_{2(m-1)n}\|_\infty\le L\cdot2\cdot\frac{11}{8}\cdot\frac{1}{2^n}.
\end{equation}
This estimate is true only for the subsequence $\Phi_{2(m-1)n}$ since we need $2(m-1)$ steps to reduce the distance between points in $B_0$ in half. We have also adjusted the quantity on the right hand side of condition (2) of Theorem \ref{thm1}. One change is the introduction of the Lipschitz constant $L$. The other change is that the length of the interval [-1,1] is two in contrast to the previous theorem, where we assumed a boundary of length one. The set of points
\begin{equation}
\{h_i^{-1}(w_{2(m-1)n,j}): \quad j=1,...,2^{2(m-1)(n+1)},\quad i=1,...,N \}
\end{equation}
 forms an $\frac{L}{2^n}$ - net on the boundary of $K$. The corresponding points\\  $h_i^{-1}(w_{2(m-1)n,j}^*)$ in $K$ satisfy condition (3) of Theorem \ref{thm1} up to the constant $2L$ on the right hand side of the inequality. Thus, the subsequence $\Phi_{2(m-1)n}$ converges uniformly to a continuous potential $\Phi$ and $\R(\Phi)=K$
\end{proof}

\section{Entropy of rotation sets}\label{press}
In this section we give an alternative definition for rotation sets and its associated entropy function which is more related to the  traditional
definition of topological entropy and does not make use of the variational principle.

Let
$(X,d)$ be a compact metric space, and let $f\colon
X\to X$ be  continuous.
 Let $\Phi=(\phi_1,\cdots,\phi_m)$ be a continuous ($m-$dimensional) potential. For $x\in X$ and $n\in \bN$ we define
the $m$-dimensional Birkhoff average   $\frac{1}{n}S_n\Phi(x)$ at $x$ of length $n$ with respect to $\Phi$ defined by
\begin{equation}\label{defSnm}
\frac{1}{n}S_n\Phi(x)=\left( \frac{1}{n}S_n\phi_1(x),\cdots,\frac{1}{n}S_n\phi_m(x)\right),
\end{equation}
where $\frac{1}{n}S_n\phi_i(x)=\frac{1}{n}\sum_{k=0}^{n-1} \phi_i(f^k(x))$.
Moreover, we define
\begin{multline}\label{defRf}
\R_{Pt}(f,\Phi)=\\
\left\{w\in \bR^m: \forall r>0\ \forall N\ \exists n\geq N\ \exists\ x\in X: \  \frac{1}{n}S_n\Phi(x)\in D(w,r)\right\}.
\end{multline}
Here $D(w,r)$ denotes the open Euclidean ball with center $w$ with radius $r$.
We frequently write $\R_{Pt}(f,\Phi)=\R_{Pt}(\Phi)$ when it is clear to which dynamical system $f$ we refer too. Recall the definition of the rotation set $\R(\Phi)$, see \eqref{defrotset}.
 We define
\begin{equation}
\R_E(\Phi)=\{\rv(\mu) :\mu\in \cM_E\}.
\end{equation}
We have the following.

\begin{proposition}\label{prop456}
Let $f:X\to X$ be a continuous map on a compact metric space and let $\Phi=(\phi_1,\cdots,\phi_m):X\to\bR^m$ be continuous. Then
\begin{equation}\label{eqinc1}
\R_E(\Phi)\subset \R_{Pt}(\Phi) \subset  \R(\Phi).
\end{equation}
\end{proposition}
\begin{proof}
Let $\mu\in \cM_E$ and $r>0$. Define $w=\rv(\mu)$. By Birkhoff's Ergodic Theorem, the basin of $\mu$
\begin{equation}
\cB(\mu)=\left\{x\in X:\ \frac{1}{n}\sum_{k=1}^{n-1} \delta_{f^k(x)} \to \mu\ {\text as}\ n\to\infty\right\},
\end{equation}
is a set of full $\mu$-measure, i.e. $\mu(\cB(\mu))=1$. By weak$*$ convergence, for all $x\in \cB(\mu)$ there exists $N=N(x)\in \bN$ such that $\frac{1}{n}S_n\Phi(x)\in D(w,r)$ for all $n\geq N$.
This proves the left-hand side inclusion in \eqref{eqinc1}.
\\
To prove the right-hand side inclusion in \eqref{eqinc1} let $w\in \R_{Pt}(\Phi)$ and consider sequences $(x_l)_{x\in \bN}\subset X$ and $n_x\in \bN, n_l\geq l$ such that
$S_{n_l}\Phi(x_l)\in D(w,\frac{1}{l})$. The existence of these sequences follows from the definition of $ \R_{Pt}(\Phi) $.
Define  probability measures
$\mu_{n_l}=\frac{1}{n_l}\sum_{k=0}^{n_l-1} \delta_{f^k(x_l)} .$ Hence $\rv(\mu_{n_l})\in D(w,1/l)$. Note that $\mu_{n_l}$ is in general not an invariant measure.
Since the space
of all Borel probability measures on $X$ is compact, there exists a Borel probability measure $\mu$ on $X$ which is a weak$\ast$  accumulation point of the measures $\mu_{n_l}$.
It now follows from similar arguments as in the proof of the Krylov-Bogolyubov Theorem  (see for example \cite{KH})  that $\mu$ is invariant. Moreover, by construction $\rv(\mu)= w$. This completes the proof.
\end{proof}

\begin{proposition}\label{prop24}
Let $f:X\to X$ be a continuous map on a compact metric space and let $\Phi=(\phi_1,\cdots,\phi_m):X\to\bR^m$ be continuous. Then
\begin{enumerate}
\item[(i)]
Both, $\R_{Pt}(\Phi)$ and $\R(\Phi)$   are compact and $\R(\Phi)$ is convex;
\item[(ii)]
If for all $w\in \R(\Phi)$ and all $r>0$ exists $\mu\in\cM_E$ such that $\rv(\mu)\in D(w,r)$ then $\R_{Pt}(\Phi) = \R(\Phi)$.
In particular, if $\overline{\cM_E}=\cM$ then  $\R_{Pt}(\Phi) = \R(\Phi)$.
\end{enumerate}
\end{proposition}
\begin{proof}
As stated before, the  weak$*$ compactness and convexity of $\cM$ implies that $\R(\Phi)$ is compact and convex.
The statement that $\R_{Pt}(\Phi)$ is closed follows directly from the definition. This proves (i).\\
Suppose for all $w\in \R(\Phi)$ and all $r>0$ exists $\mu\in\cM_E$ such that $\rv(\mu)\in D(w,r)$. Then $\overline{\R_E(\Phi)}=\R(\Phi)$, and therefore (ii) is a consequence of \eqref{eqinc1} and (i).
\end{proof}
 \noindent
{\rm Remark.\ } As a consequence of Proposition \ref{prop24} (ii)
we obtain that $\R_{Pt}(\Phi) = \R(\Phi)$ for topological mixing subshifts of finite type, Axiom A basic sets and expansive systems with specification since  in these cases
$\overline{\cM_E}=\cM$ holds.
\\[0.2cm]
\noindent
The following examples show that both of the inclusions \eqref{eqinc1} can be strict.

\begin{example}\label{ex1}
Let $a,b,c,d\in \bR$ with $a<b<c<d$. Let $X=[a,b]\cup [c,d]$ and $f:X\to X$ be continuous with $f([a,b])\subset [a,b]$ and $f([c,d])\subset [c,d]$.
Moreover, we assume that $f(a)=a$ and $f(d)=d$. Consider the potential $\Phi={\rm id}_X$. The fact that $[a,b]$ and $[c,d]$ are both invariant sets of $f$ implies that $ \R_{Pt}(\Phi)\subset X$. On the other hand, since $\delta_a,\delta_d\in \cM$ the  convexity of $\R(\Phi)$ implies $\R(\Phi)=[a,d]$.
In particular,  the inclusion $\R_{Pt}(\Phi) \subset  \R(\Phi)$ is strict.
\end{example}
\begin{example}\label{ex2}
Let $d\geq 6$ be even and  let $f:X\to X$ be the one-sided full shift  on the alphabet $\{0,\cdots, d-1\}$.
Let $K$ be a compact and convex subset of $\bR^2$ whose boundary is a closed polygon with vertices  $w_1,\cdots, w_{d/2}$ which we label counter clock-wise. Since there are at least $3$ vertices $w_i$ the set $K$ has non-empty interior.
Pick $w_0$ in the interior of $K$.
Let $l_1,\cdots, l_{d/2}$ be the line segments joining $w_0$ and $w_i$ endowed with the canonical order induced by  $w_0<w_i$.
For each $i=1,\cdots, d/2$ we pick a strictly  increasing sequence $(w_i(k))_{k\in \bN}\subset {\rm int}\ l_i$ with $|w_i(k)-w_i|<1/2^k$, in particular $\lim_{k\to \infty} w_i(k)=w_i$.
We also write $w_i(0)=w_0$.

Next, we define several subsets of $X$.
Let $S_1=\{0,1\}, \cdots, S_{d/2}=\{d-2,d-1\}$ and fix $ \alpha\in \bN, \alpha\geq 3$.
For all $i=1,\cdots, d/2$ and all $k\geq \alpha$ we define $X_i(k)=\{x \in X: x_1,\cdots, x_{k}\in S_i\}$. Moreover, let $X_0(\alpha)= X\setminus (X_1(\alpha)\cup \cdots \cup X_{d/2}(\alpha))$.
\\[0.2cm]
\noindent
Finally, we define a potential $\Phi: X\to \bR^2$ by
\begin{equation}
\Phi(x)=\begin{cases}
w_0\qquad & if\,\,   x\in X_0(\alpha)\\
                         w_{i}(k-\alpha) & if\,\,  x\in X_{i}(k)\ {\rm and}\ x\not\in X_i(k+1)\\
                          w_i & if\,\,   \  x\in X_{i}(k)\  {\rm for\, all }\, k\geq \alpha
            \end{cases}
\end{equation}
Note that $\Phi(x)=w_i$ if and only if $x_k\in \{2i-2,2i-1\}$ for all $k\in \bN$, in particular $f|_{\Phi^{-1}(w_i)}$ is conjugate to a full shift in $2$ symbols.
\end{example}
\begin{figure}[h]
\begin{center}
\psfragscanon
\psfrag{0}[c][l]{\LARGE{$w_0$}}
\psfrag{1}[c][l]{\LARGE{$w_1$}}
\psfrag{2}[c][l]{\LARGE{$w_2$}}
\psfrag{3}[c][l]{\LARGE{$w_3$}}
\psfrag{4}[c][l]{\LARGE{$w_4$}}
\psfrag{5}[c][l]{\LARGE{$w_5$}}
\psfrag{k}[l][l]{\LARGE{$\{w_1(k)\}$}}

\scalebox{0.65} {
\includegraphics{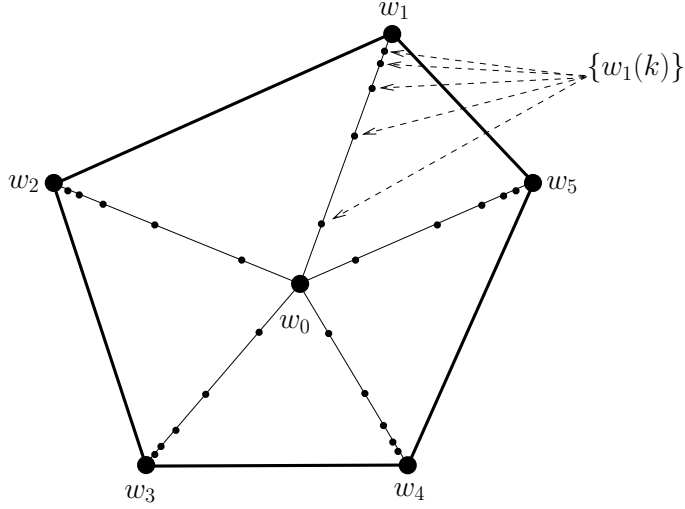}
}
\caption{This figure illustrates Example \ref{ex2}}
\end{center}
\end{figure}

We now  list several properties of the system in example \ref{ex2}.
\begin{theorem}
Let $X,f$ and $\Phi$ be as in example $2$. Then
\begin{enumerate}
\item[(i)] $\Phi$ is Lipschitz continuous;
\item[(ii)] $\R(\Phi)=K$;
\item[(iii)]  $\R_{Pt}(\Phi)=\R(\Phi)$;
\item[(iv)] $\R_E(\Phi)={\rm int }\ K\cup \{w_1,\cdots, w_{d/2}\}$, in particular $$\R_E(\Phi)\cap \left(\partial K\setminus \{w_1,\cdots, w_{d/2}\}\right)=\emptyset,$$ and therefore
the inclusion $\R_E(\Phi)\subset \R_{Pt}(\Phi)$ is strict;
\item[(v)] $H({\rm int }\ K)=(\log 2, \log d]$ and $H(\partial K)=\log 2$.
\end{enumerate}
\end{theorem}
\begin{proof}
{\rm (i) } We will work with the $d_{1/2}$ metric (see \eqref{defmet}) on $X$ to show that $\Phi$ is Lipschitz continuous. Set $C={\rm diam}(K)$. Let $x,y\in X$ with $\Phi(x)\not=\Phi(y)$. If $\Phi(x)=w_0$
then $x_k\not=y_k$ for some $k\leq \alpha$. Hence,
\begin{equation}\label{eqwsx}
\|\Phi(x)-\Phi(y)\|_2\leq C = C 2^\alpha \frac{1}{2^\alpha} \leq C 2^\alpha d(x,y)
\end{equation}
The case $\Phi(y)=w_0$ is analogous. The case $\Phi(x)\in l_i\setminus \{w_0\}$ and $\Phi(y)\in l_j\setminus \{w_0\}$ with $i\not=j$ can be treated analogously as in \eqref{eqwsx}.
It remains to consider the case $\Phi(x),\Phi(y)\in l_i\setminus \{w_0\}$ for some $i=1,\cdots,d/2$. Without loss of generality we assume $\Phi(x)<\Phi(y)$ (with respect the canonical order on $l_i$.
Thus, $d(x,y)\geq\frac{1}{2^{k_x+1}}$ where $k_x$ is defined by $\Phi(x)=w_i(k_x)$. Since $|w_i(k)-w_i|<1/2^k$ for all $k \in \bN$ we conclude that
\begin{equation}\label{eqwsx1}
\|\Phi(x)-\Phi(y)\|_2\leq \frac{1}{2^{k_x}} \leq 2 d(x,y),
\end{equation}
which completes the proof of (i).\\
\noindent
(ii) Since $\Phi(X)\subset K$ and since $K$ is compact and convex, it follows that $\R(\Phi)\subset K$. We have $\{w_1,\cdots,w_{d/2}\}\subset K$. This follows from the fact that for each $i$ there exists a fixed point $x(i)\in X$
with $\Phi(x(i))=w_i$. Using that $w_1,\cdots,w_{d/2}$ are the extreme points of $K$ and that $\R(\Phi)$ is compact and convex the inclusion $K\subset \R(\Phi)$ follows from the Krein-Milman theorem.\\
\noindent
(iii)
By a result Sigmund (see \cite{DGS}) applied to one-sided full shifts the set of periodic point measures (see \eqref{defpermes} for the definition) is weak$*$ dense in $\cM$. Therefore, (iii) follows from Proposition
\ref{prop24} (ii). \\
\noindent
(iv) We already have shown in the proof of (ii) that $\{w_1,\cdots,w_{d/2}\}\subset \R_E(\Phi)$. The statement ${\rm int }\ K \subset \R_E(\Phi)$ will be proven in Corollary \ref{corintemp}
by using the thermodynamic formalism. Consider $w\in \partial K\setminus \{w_1,\cdots,w_{d/2}\}$. To prove (iv) we have to show that $w\not\in \R_E(\Phi)$. Let $i,j\in \{1,\cdots,d/2\}$ such that $w$ lies on the interior of the line segment $[w_i,w_j]$ joining $w_i$ and $w_j$. By construction of $K$ the line segment $[w_i,w_j]$ is a face of $K$ and since $\Phi(X)\cap (w_i,w_j)=\emptyset$ each invariant measure $\mu$ with $\rv(\mu)=w$ must put positive measure on each of the $f$-invariant sets $\Phi^{-1}(w_i)$ and $\Phi^{-1}(w_j)$. This implies that
$\mu$ is not ergodic.\\
\noindent
(v) Clearly, $H(w_i)=\log 2$ for all $i\in \{1,\cdots,d/2\}$. Let now $w\in \partial K\setminus \{w_1,\cdots,w_{d/2}\}$. It follows
from a similar argument as in the proof of (iv) that each $\mu\in \cM$ with $\rv(\mu)=w$ must be a convex combination of invariant measures $\mu_1,\mu_2$ both of which have rotation vectors
in $\{w_1,\cdots,w_{d/2}\}$. Therefore, $H(\partial K)=\log 2$ follows from the fact that the measure-theoretic entropy is affine.
Next, we consider $w\in {\rm int }\ K$. Let $\mu_0\in \cM$ be the unique measure of maximal entropy of $f$, i.e., the unique invariant measure satisfying $h_{\mu_0}(f)=\log d$. It follows from the construction
that $\rv(\mu_0)=w_0$. Therefore, it suffices to consider $w\not=w_0$.
Let $\tilde{w}$ be the unique point on $\partial K$ such that $w$ lies on the interior of the line segment $[w_0,\tilde{w}]$. Let $\tilde{\mu}\in\cM$ such that $\rv(\tilde{\mu})=\tilde{w}$
and $h_{\tilde{\mu}}(f)=\log 2$. Let $t\in (0,1)$ such that  for $\mu=t\mu_0+(1-t)\tilde{\mu}$ we have $\rv(\mu)=w$. Since the measure-theoretic entropy is affine we conclude that $h_{\mu}(f)>\log 2$.
Finally,  $H({\rm int }\ K)=(\log 2, \log d]$ follows from the continuity of $w\mapsto H(w)$ and from the fact that ${\rm int}\ K$ is connected.
\end{proof}

\noindent
{\it  Remark. } By property {\rm (v)}, the entropy $H(w)$ is on the boundary of $K$ strictly smaller than in the interior of $K$. This is however in general not true. Indeed, by slightly modifying
example $\rm 2$ and concentrating a shift with more then $2$ symbols on one of the points $w_i$, the continuity of $w\mapsto H(w)$ implies that $\max_{w\in \partial K} H(w)>\inf_{w\in {\rm int } K} H(w)$. \\[0.2cm]
\noindent
Next, we introduce an alternative definition for a entropy function which is closely related to the traditional definition of topological entropy.
Recall the definition of the \emph{topological entropy} of $f$ (see \eqref{defdru}),
\begin{equation}\label{defent}
  h_{\rm top}(f) \eqdef \lim_{\varepsilon \to 0}
            \limsup_{n\to \infty}
            \frac{1}{n} \log  \car \ F_n(\epsilon),
\end{equation}
where $ F_n(\epsilon)$ is a maximal $(n,\epsilon)$-separated set.
The topological entropy satisfies the variational principle (which is a special case of the variational principle for the topological pressure \eqref{eqvarpri}):
\begin{equation}\label{eqvarprient}
h_{\rm top}(f)= \sup_{\mu\in \cM} h_\mu(f).
\end{equation}
Furthermore, the supremum in~\eqref{eqvarprient} can be replaced by
the supremum taken only over all $\mu\in\cM_{\rm E}$. We denote by
$E_{\rm max}(f)$ the set of all measures of maximal entropy, that is the set of measures
$\mu\in\cM$ which attain the supremum in~\eqref{eqvarpri}. In general $E_{\rm max}(f)$ may be empty (see for example \cite{Mi2}).

Fix $w\in \R_{Pt}(\Phi)$.
Let $n\in \bN$ and $\epsilon,r>0$.  We say $F\subset X$ is  a $(n,\epsilon,w,r)$-set if $F$ is $(n,\epsilon)$-separated and
$\frac{1}{n}S_n\Phi(x)\in D(w,r)$ for all $x\in F$. For all $n\in \bN$ and $\epsilon,r>0$  we pick a maximal (with respect to the inclusion)
$(n,\epsilon,w,r)$-set $F_n(\epsilon,w,r)$ and define
\begin{equation}\label{defhw}
h(\epsilon,w,r)=\limsup_{n\to\infty}\frac{1}{n} \log \car \ F_n(\epsilon,w,r)
\end{equation}
and
\begin{equation}\label{eqdefhw234}
h(w)=\lim_{r\to 0} \lim_{\epsilon\to 0} h(\epsilon,w,r)
\end{equation}
Analogous as in the case of $h_{\rm top}(f)$ one can show that $h(w)$ does not depend on the choice of the $(n,\epsilon,w,r)$-sets $F_n(\epsilon,w,r)$. Clearly, $h(w)$ and $h(\epsilon,w,r)$ are bounded above by $h_{\rm top}(f)$ and therefore finite.

We now review a standard construction of invariant measures with large entropy.
Given a finite set $F\subset X$ we  define  a probability measure $ \sigma(F)$ by
\begin{equation}\label{tilsig}
 \sigma(F)=
\frac{1}{\car\ F}
\sum_{x\in F} \delta_x.
\end{equation}

Recall that for a Borel map $g:X\to X$ and a Borel measure $\mu$ on $X$ the push forward of $\mu$ is defined by $g_*\mu(A)=\mu(g^{-1}(A))$.

We will need the  following result, which is typically shown by using the Misiurewicz argument when proving the variational principle (see, for example,~\cite[Section 4.5]{KH}).

\begin{lemma}\label{lemKH}
 Let $f:X\to X$ be a continuous map on a compact metric space and let $\epsilon>0$.
 Let $(F_n)_{n\in \bN}$ be a sequence of  $(n,\epsilon)$-separated sets in $X$, and define the measures
\begin{equation}\label{defmuin}
\nu_n= \sigma(F_n),
\quad
\mu_n=\frac{1}{n}\sum_{k=0}^{n-1}f^k_*\nu_n.
\end{equation}

Then there exists a weak$\ast$ accumulation point $\mu$ of the measures $(\mu_n)_{n\in\bN}$ and any such accumulation point $\mu$ is invariant  and satisfies
\begin{equation}\label{lemmis}
 \limsup_{n\to \infty}\frac{1}{n}\log\sum_{x\in F_n}\car\ F_n
\leq h_\mu(f).
\end{equation}
\end{lemma}

The next result establishes an inequality between $h(w)$ and $H(w)$.

\begin{proposition}\label{propungleich}
Let $f:X\to X$ be a continuous map on a compact metric space, let $\Phi=(\phi_1,\cdots,\phi_m):X\to \bR^m$ be continuous and let $w\in \R_{Pt}(\Phi)$.
Suppose $H$ is continuous at $w$.
Then $h(w)\leq H(w)$.
\end{proposition}
\begin{proof}
Note that $H(w)$ is well-defined by \eqref{eqinc1}. Let $\eta>0$. It follows from the definition of $h(w)$ (see \eqref{eqdefhw234}) and from
the continuity of $H$ at $w$ that there exists $r^*=r^*(\eta)>0$ such that for all $0<r\leq r^*$ and all $v\in D(w,r^*)$ we have
\begin{equation}\label{eqqaz}
h(w)\leq \lim_{\epsilon\to 0} h(\epsilon,w,r)<h(w)+\frac{\eta}{2}
\end{equation}
and
\begin{equation}\label{eqqaz1}
H(w)-\frac{\eta}{2}<H(v) <  H(w)+\frac{\eta}{2}.
\end{equation}
Using \eqref{eqqaz} and the definition of $h(\epsilon,w,r)$ we can pick $0<r<r^*$, $\epsilon>0$  and an increasing sequence $(n_i)_{i\in\bN}$ of positive integers such that
\begin{equation}
h(w)-\frac{\eta}{2}<\limsup_{i\to\infty}\frac{1}{n_i} \log \car \ F_{n_i}(\epsilon,w,r)<h(w)+\frac{\eta}{2}.
\end{equation}
We now apply Lemma \ref{lemKH} to  the $(n_i,\epsilon)$-separated sets $F_{n_i}(\epsilon,w,r)$ and obtain  the existence of $\mu\in \cM$ with $h(w)-\eta/2\leq h_{\mu}(f)\leq H(\rv(\mu))$ and $\rv(\mu)\in D(w,r)$. It follows that $h(w)-\eta<H(w)$ and since $\eta>0$ was arbitrary the claim follows.
\end{proof}

\noindent
\begin{remark}
We recall that the continuity of $w\mapsto H(w)$ holds for all $w\in \R(\Phi)$ if the entropy map $\mu\mapsto h_\mu(f)$ is upper semi-continuous. In particular, this is true, if $f$ is expansive \cite{Wal:81} or if $f$ is a $C^\infty$ map on a smooth Riemannian manifold \cite{N}.
\end{remark}

\noindent
In the following we establish  under rather mild assumptions the identity of $h(w)$ and $H(w)$.

We say a topological space  is a Besicovitch space if the Besicovitch covering theorem holds. Next, we give an alternative
formulation for the measure-theoretic entropy which is due to Katok.
Let $f:X\to X$ be a continuous map on a compact metric space and
let $\mu\in \cM_E$. For $n\in \bN$, $\epsilon>0$ and $0<\delta<1$
we denote by $N(n,\epsilon,\delta)$ the minimal number of $\epsilon$ balls in the $d_n$ metric that cover a set of measure greater or equal than $1-\delta$.
It is shown in \cite{Kat:80} that
\begin{equation}\label{katent}
h_\mu(f)=\lim_{\epsilon\to 0}\liminf_{n\to \infty} \frac{\log N(n,\epsilon,\delta)}{n}  = \lim_{\epsilon\to 0}\limsup_{n\to \infty} \frac{\log N(n,\epsilon,\delta)}{n}.
\end{equation}
Given a continuous potential $\Phi$  and  $w\in \R(\Phi)$ we say that $H(w)$ is \emph{approximated by ergodic measures} if there exists $(\mu_n)_{n\in\bN}\subset \cM_E$ such that $\rv(\mu_n)\to w$ and $h_{\mu_n}(f)\to H(w)$ as $n\to\infty$.  In this case we have $w\in \R_{Pt}(\Phi)$ (see Proposition \ref{prop24} {\rm (ii)}).
Being approximated by ergodic measures occurs for several classes of systems and potentials.
For example, we will show in  Corollary \ref{corintemp} that if $f$ satisfies \STP  and $\Phi$ is H\"older continuous then
$H(w)$ can be approximated by ergodic measures for all $w\in \R(\Phi)$.
We are now ready to state our main result about the identity of $h(w)$ and $H(w)$.

\begin{theorem}\label{thhwHw}
Let $f:X\to X$ be a continuous map on a compact metric space $X$ that is a Besicovitch space with respect to the induced topology.  Let $\Phi=(\phi_1,\cdots,\phi_m):X\to\bR^m$ be continuous
and let $w\in \R(\Phi)$ such that $H$ is continuous at $w$ and $H(w)$ is approximated by ergodic measures. Then $h(w)=H(w)$.
\end{theorem}

\begin{proof}

The inequality $h(w)\leq H(w)$ is shown in Proposition \ref{propungleich}. To show $H(w)\leq h(w)$ we only need to consider the case $H(w)>0$.
Let $\eta>0$ be arbitrary.
Pick $r_0>0$ such that
\begin{equation}
\left|\lim_{\epsilon\to 0} h(\epsilon,w,r)- h(w)\right|< \frac{\eta}{2}
\end{equation}
for all $0<r\leq r_0$.
  It therefore suffices
to show that there exists $\epsilon_0>0$ such that
\begin{equation}
 h(\epsilon,w,r_0)>H(w)-\frac{\eta}{2}
\end{equation}
for all $0<\epsilon\leq\epsilon_0$.
 By uniform continuity of $\Phi$ there exists $\epsilon_0$ such that if $0<\epsilon\leq \epsilon_0$,  $x\in X$ and $n\in \bN$ then for all
$x_1,x_2\in B_n(x,\epsilon)$ we have
\begin{equation}\label{varpot}
\left|\frac{1}{n}S_n\Phi(x_1)-\frac{1}{n}S_n\Phi(x_2)\right|<r_0/4.
\end{equation}
 Here $\frac{1}{n}S_n\Phi$ denotes the $m$-dimensional
Birkhoff average defined in \eqref{defSnm}.
Since $H(w)$ is approximated by ergodic measures there exists $\mu\in \cM_E$ such that $|\rv(\mu)-w|<r_0/2$ and $h_\mu(f)>H(w)-\eta/2$.
Let $0<\delta<1$ be fixed. Applying \eqref{katent} and making  $\epsilon_0$ smaller if necessary we conclude that
\begin{equation}\label{eqkatb}
\liminf_{n\to \infty} \frac{\log N(n,\epsilon,\delta)}{n} > H(w)-\eta/2.
\end{equation}
for all $0<\epsilon\leq\epsilon_0$. From now we consider a fixed $0<\epsilon\leq\epsilon_0$. We  define
\begin{equation}\label{defBn}
\cB_{n,r_0/4}(\mu)=\left\{x\in \cB(\mu): |S_l\Phi(x)-\rv(\mu)|<r_0/4\, \ {\rm for\  all }\, \ l\geq n\right\}.
\end{equation}
Since $(\cB_{n,r_0/4}(\mu))_{n\in \bN}$ is an increasing sequence of Borel sets whose union is a set of full $\mu$-measure we conclude that
\begin{equation}\label{eqvbn}
\lim_{n\to \infty} \mu(\cB_{n,r_0/4}(\mu))=1.
\end{equation}
We denote by $\tilde{N}(n,\epsilon,\delta)$ the number of Bowen balls determining $N(n,\epsilon,\delta)$ that have non-empty intersection with  $\cB_{n,r_0/4}(\mu)$. For all sufficiently large $n$ these balls, due \eqref{defBn}, will cover a set of measure $1-\delta'$ for some $\delta<\delta'<1$.
It follows from the fact that \eqref{katent} also holds for $\delta'$ that \eqref{eqkatb} remains true when we replace $N(n,\epsilon,\delta)$ by
$\tilde{N}(n,\epsilon,\delta)$.
For $t\in \bR$ let $[t]$ denote the largest integer smaller or equal than $t$.
Let $\beta$ be a Besicovitch constant of $X$. Note that this constant can be chosen independently of the metrics $d_n$ since they are decreasing in $n$.
It follows from the Besicovitch covering theorem that there exist $[\tilde{N}(n,\epsilon,\delta)/\beta]$ Bowen balls
in the collection of balls determining $\tilde{N}(n,\epsilon,\delta)$ that are pair-wise disjoint. The centers of these balls form a $(n,\epsilon)$-separated set which we denote by
$F_n(\epsilon,\delta)$. It follows from \eqref{eqkatb} and the construction of the set   $F_n(\epsilon,\delta)$ that
\begin{equation}\label{eqh(w)12}
\liminf_{n\to \infty} \frac{\log \car \  F_n(\epsilon,\delta)}{n} > H(w)-\eta/2.
\end{equation}
Moreover, since each of the balls defining $\tilde{N}(n,\epsilon,\delta)$ has non-empty intersection with $\cB_{n,r_0/4}(\mu)$ we obtain from \eqref{varpot} and \eqref{defBn} that
$\frac{1}{n}S_n\Phi(x)\in D(w,r)$ for all $x\in F_n(\epsilon,\delta)$. Finally,  we may conclude from \eqref{eqh(w)12}  (also using the definition of $h(\epsilon,w,r_0)$, see \eqref{defhw}) that $h(\epsilon,w,r_0)>H(w)-\eta/2$ which completes the proof of the theorem.
\end{proof}

\section{Entropy via periodic orbits}

In this section we consider smooth non-uniformly expanding systems and show that under certain assumptions on $w\in \R(\Phi)$ the entropy $H(w)$ is entirely determined by the growth rate of those periodic orbits whose rotation vectors are sufficiently
close to $w$. Our approach heavily relies on previous work by Gelfert and the second author of this paper \cite{GW2} about the computation of the topological pressure via
periodic orbits. Throughout this section we use the notations from Section 2.5.

Let $M$ be a smooth Riemannian manifold and let
$f\colon M\to M$ be a $C^{1+\epsilon}$-map. Let $X\subset M$ be compact locally maximal $f$-invariant set. We consider $f|_X$ and simply write $f$ instead of  $f|_X$.
Let $\Phi=(\phi_1,\cdots,\phi_m): X\to \bR^m$ be a continuous potential  with rotation set $\R(\Phi)$.
To avoid trivialities we will always assume $h_{\rm top}(f)>0$.

We say that $H(w)$ is uniformly approximated by    measures in $\cM^+_E$ if there exist $\chi(w)>0$ and $(\mu_k)_{k\in\bN}\subset \cM_E^+$ such that $\chi(\mu_k)\geq \chi(w)$ for all $k\in\bN$, and $\rv(\mu_k)\to w$ as well as $h_{\mu_k}(f)\to H(w)$ as $k\to\infty$.

We now introduce an entropy-like quantity which is entirely defined  by the growth rate of those  periodic points that have rotation vectors in a given ball about $w$ and have some uniform expansion.

Let $w\in \R(\Phi),r>0$ and let
$0<\alpha$, $0<c\leq 1$. Define
\begin{equation}
 h_{\rm per}(w,r,\alpha,c,n) =
\car \ \Per_n(w,r,\alpha,c)
\end{equation}
if $\Per_n(w,r,\alpha,c)\ne\emptyset$ and
\begin{equation}\label{hper}
h_{\rm per}(w,r,\alpha,c,n)= 1
\end{equation}
otherwise. Furthermore, we define
\begin{equation}\label{hper}
h_{\rm per}(w,r,\alpha,c) = \limsup_{n\to\infty}
                 \frac{1}{n}\log h_{\rm per}(w,r,\alpha,c,n).
\end{equation}
We have the following.

\begin{proposition}\label{li}
  Let $w\in \R(\Phi)$  and suppose $H(w)$ is uniformly approximated by   measures in $\cM^+_E$.  Let $r>0$.
Then for all $0< \alpha< \chi(w)$ we have
  \begin{equation}\label{eqas}
  H(w) \le
  \lim_{c\to 0} h_{\rm per}(w,r,\alpha,c).
  \end{equation}
\end{proposition}

\begin{proof}
If $H(w)=0$ the statement is trivial; therefore, we can assume $H(w)>0$.
Let $0<\alpha<\chi(w)$ and let $\delta>0$. Since $H(w)$ is uniformly approximated by  measures in $\cM^+_E$ there exists $\mu\in \cM_E^+$
with  $\chi(\mu)>\alpha$, $\rv(\mu)\in D(w,r/2)$ and $|h_\mu(f)-H(w)|<\delta$.

 It is a consequence of Katok's theory of approximation of hyperbolic measures by hyperbolic sets in it's version
for non-uniformly expanding maps (see for example \cite{GW2} and the references therein) that there exists
a sequence $(\mu_n)_n$ of measures $\mu_n\in \cM_{\rm E}$ supported on compact
invariant expanding sets $X_n\subset X$ such that
\begin{equation}\label{holl}
h_\mu(f) \le \liminf_{n\to\infty}h_{\rm top}(f|_{X_n}),
\end{equation}
$\mu_n\to\mu$ in the weak$\ast$ topology and  that $\rv(x)\in D(w,r)$ for all $x\in \Per(f)\cap X_n$.
Furthermore, for each $n\in\bN$ there exist $l,s\in\bN$ such that $f^l|_{ X_n}$
is conjugate to the full shift in $s$ symbols. For every
$\eta>0$ there is a number $N=N(\varepsilon)\ge 1$
such that
\begin{equation}\label{kuh}
h_\mu(f) -\eta
\le h_{\rm top}(f|_{X_n})
\end{equation}
for all $n\geq N$.
Moreover, there exists a number $c_0=c_0(n)$ with $0<c_0(n)\le1$ such that for
every  $x\in X_n$ with $x\in \Per_k(f)$ we have $x\in \Per_k(\alpha,c_0)$.
Together we obtain
\begin{equation}\label{kir}
  \Per_k(f)\cap X_n \subset \Per_k(w,r,\alpha,c_0)
\end{equation}
for every $k\in\bN$.
Let $l, s\in\bN$ such that $f^l|_{X_n}$ is topologically conjugate to the full
shift in $s$ symbols.
Since $l\cdot h_{\rm top}(f|_{X_n}) = h_{\rm top}(f^l|_{X_n})$
(see~\cite[Theorem 9.8]{Wal:81}), we may conclude that
\begin{equation*}
  h_\mu(f)-\eta
  \le \frac{1}{l}h_{\rm top}(f^l|_{X_n}).
\end{equation*}
It now follows from Proposition~\ref{ha} and an elementary calculation that
\begin{equation}\label{eqrep}
  \begin{split}
  &h_\mu(f) -\eta\\
    &\le
    \frac{1}{l}\limsup_{k\to\infty}\frac{1}{k}\log\left(\car  \left(\Per_{lk}(f)\cap X_n\right)
      \right)\\
  &\le \limsup_{k\to\infty}\frac{1}{k}\log\left(\car  \left(\Per_{k}(f)\cap X_n \right)\right).
\end{split}
\end{equation}
Combining~\eqref{kir} and~\eqref{eqrep} yields
\begin{equation*}
  h_\mu(f)-\eta
  \le \limsup_{k\to\infty}\frac{1}{k}\log\left( \car \ \Per_k(w,r,\alpha,c_0)\right).
\end{equation*}
Recall that by~\eqref{ni} the  map $c\mapsto h_{\rm per}(w,r,\alpha,c)$ is
non-decreasing as $c\to 0$.
Since $\eta>0$ and $\delta>0$ can be chosen arbitrarily small the claim follows.
\end{proof}

The following Theorem  is the main result of this section.

\begin{theorem}\label{theoperrot}
Let $f:M\to M$ be a $C^{1+\epsilon}$-map, and let $X\subset M$ be a compact $f$-invariant locally maximal set. Let $\Phi=(\phi_1,\cdots,\phi_m): X\to \bR^m$ be continuous and
let $w\in\R(\Phi)$ such that $H(w)$ is uniformly approximated by   measures in $\cM^+_E$ and that $H$ is continuous at $w$. Then for all $0<\alpha<\chi(w)$,

\begin{equation}\label{idmain}
H(w)=\lim_{r\to 0}\lim_{c\to 0}\limsup_{n\to\infty} \frac{1}{n} \log  h_{\rm per}(w,r,\alpha,c,n).
\end{equation}

\end{theorem}

\begin{proof}
  The $"\leq"$ part in \eqref{idmain}  already follows from Proposition \ref{li}. To prove the opposite inequality
pick $0<\alpha<\chi(w)$ and $\eta>0$.
Since the right hand side in the limit $r\to 0$ in \eqref{idmain}  is non-increasing as $r\to 0$, it suffices to show that
\begin{equation}\label{idmain2}
\lim_{c\to 0} h_{\rm per}(w,r,\alpha,c)< H(w)+\eta
\end{equation}
for some $r>0$.
By continuity of $H$ at $w$ there exists $r>0$ such that $H(v)<H(w)+\eta/2$ for all $v\in D(w,r)\cap \R(\Phi)$. Pick $c_0>0$ such that
\begin{equation}\label{main4}
\lim_{c\to 0}\limsup_{n\to\infty} \frac{1}{n} \log h_{\rm per}(w,r,\alpha,c,n)-\frac{\eta}{2}<
\limsup_{n\to\infty} \frac{1}{n} \log h_{\rm per}(w,r,\alpha,c_0,n)
\end{equation}
for all $0<c\leq c_0$.
If  $\Per_n(w,r,\alpha,c_0)=
\emptyset$ for all $n\in\bN$, then   $h_{\rm per}(w,r,\alpha,c_0)=0$ and \eqref{idmain2} trivially holds.
Next we consider the case that $\Per_n(w,r,\alpha,c_0)\not=
\emptyset$ for some $n\in\bN$.
We define
\[
X_{w,r,\alpha,c_0}
\eqdef \overline{\bigcup_{n=1}^\infty \Per_n(w,r,\alpha,c_0)} .
\]
It follows from the continuity of the  minimums norm of $Df$ that $X_{w,r,\alpha,c_0} $ is a  compact invariant uniformly expanding  set for $f$.\footnote{This implies that limit $c\to 0$ on the
left-hand side of \eqref{main4} is bounded above by $h_{\rm top}(f)$ and thus in particular finite.}
For $n\ge 1$ with $\Per_n(w,r,\alpha,c_0)\ne\emptyset$ we
define a measure $ \sigma_n= \sigma_n(w,r,\alpha,c_0)$ by
\begin{equation}\label{tilsig}
 \sigma_n=
\frac{1}{\car\  \Per_n(w,r,\alpha,c_0)}
\sum_{x\in\Per_n(w,r,\alpha,c_0)}  \delta_x,
\end{equation}
where $\delta_x$ denotes the delta Dirac measure supported at $x$.
Note that every measure $ \sigma_n$ defined
in~\eqref{tilsig} belongs to $\cM$ and is in the convex hull of the set
$\{\delta_x\colon x\in \Per_n(w,r,\alpha,c_0)\}$.
Since $\cM$ is weak$\ast$ compact, there exists a subsequence $( \sigma_{n_k})_k$ converging to some measure
$\mu=\mu_{w,r,\alpha,c_0}\in\cM$ in the weak$\ast$ topology.
It follows that $\chi(\mu)\ge \alpha$ and $\rv(\mu)\in D(w,r)$.
Since $X_{w,r,\alpha,c_0} $ is uniformly expanding, there exists $\delta=\delta(w,r,\alpha,c_0)$ which is an
expansivity constant for $f|_{X_{w,r,\alpha,c_0} }$. In
particular, for every $n\in \bN$ and every  $0<\varepsilon'\le\delta$ the set
$\Per_n(w,r,\alpha,c_0)$ is
$(n,\varepsilon')$-separated. It now follows from Lemma \ref{lemKH} that
\begin{equation}\label{krey}
  \limsup_{n\to\infty}\frac{1}{n}\log \car \ \Per_n(w,r,\alpha,c_0) \le
  h_{\mu}(f)\le H(\rv( \mu))<H(w)+\frac{\eta}{2}.
\end{equation}
Combining this with \eqref{main4} gives \eqref{idmain2} and the proof is complete.
\end{proof}

\noindent
\begin{remarks} {\rm (i)} While Theorem \ref{theoperrot} is stated in the context of non-uniformly expanding systems the analogous result holds for non-uniformly hyperbolic diffeomorphisms of saddle type. For such a map one defines $\chi(\mu)$ as the
smallest absolute value of the Lyapunov exponents of $\mu$. We refer to \cite{GW1} for more details about these classes of systems.\\
{\rm (ii)} We recall that the continuity of $w\mapsto H(w)$  holds for all $w\in \R(\phi)$
if the entropy map $\mu\mapsto h_\mu(f)$ is upper-semi continuous (and therefore particular  if $f|_X$ is expansive).\\
{\rm (iii)} If $X$ is  a topological mixing hyperbolic set (uniformly expanding or of saddle type) then it can be shown that Theorem \ref{theoperrot} holds for all H\"older continuous potentials $\Phi$ and all $w\in \R(\Phi)$.
Similar results hold for topological mixing subshifts of finite type and expansive homeomorphisms with specification.
\\
{\rm (iv)} Theorem \ref{theoperrot} even provides new information in the case when $X$ is a topological mixing hyperbolic set and $w$ is the rotation vector of the measure of maximal entropy. Indeed,  Proposition \ref{ha} provides a version of \eqref{idmain}  by considering all periodic points. On the other hand, our result shows that it is already sufficient to consider periodic points with sufficiently “large” Lyapunov exponents and rotation vectors sufficiently close to $w$.
\end{remarks}

\section{Dependence on parameters}
Let $f:X\to X$ be a continuous map on a compact metric space. In this section we assume that $f$ has strong thermodynamic properties \STP. Let $\Phi=(\phi_1,\cdots,\phi_m)$, where $\phi_1,\cdots,\phi_m\in C^\alpha(X,\bR)$ for some fixed $\alpha>0$.

The main goal of this section is to study the dependence of the entropy $H(w)$  on $w$ in the interior of the rotation set.
Recall that since the entropy map is upper semi-continuous (property (2) in the definition of \STP) it follows that $w\mapsto H(w)$ is continuous. Here we show
that under the assumption of strong thermodynamic properties  $w\mapsto H(w)$
 is even real analytic on the interior of $\R(\Phi)$.
We start by introducing some notation.

Given $T=(t_1,\cdots,t_m)\in \bR^m$ we consider the linear combination $T\cdot \Phi = t_1\phi_1+\cdots + t_m\phi_m$ of the potentials $\phi_1,\cdots,\phi_m$.
We write
\begin{equation}
Q(T)=P_{\rm top}(T\cdot \Phi).
\end{equation}
It follows from property 3 of \STP\  that $Q$ is a real-analytic function of $\bR^m$.
Let $\mu_T$ denote  the unique equilibrium measure of the potential $T\cdot \Phi$ (which is well-defined by property 4 of \STP).
As a  consequence of \eqref{eqdifpre} we obtain that
\begin{equation}\label{eqw1}
DQ(T)=\rv(\mu_T).
\end{equation}
Hence, the map $T\mapsto \rv(\mu_T)$ is also real-analytic.
Writing $h(T)=h_{\mu_T}(f)$ gives
\begin{equation}\label{eqvarT}
Q(T)=h(T)+ T\cdot DQ(T)
\end{equation}
which implies that $T\mapsto h(T)$ is also real-analytic.
First we apply results in \cite{Je} to obtain a characterization for $\R(\Phi)$ having non-empty interior.

\begin{proposition}\label{proptri}
The  following are equivalent.
\begin{enumerate}
\item[(i)]
No non-trivial linear combination of $\Phi$ is cohomologous to a constant.
\item[(ii)]
${\rm int}\  \R(\Phi)\not=\emptyset$.
\end{enumerate}
\end{proposition}
\begin{proof}
Assume that non-trivial linear combination of $\Phi$ is cohomologous to a constant. It follows from \eqref{gg33}  that $Q$ is a strictly convex function in $\bR^m$.
Therefore, \cite[Corollary 3]{Je} implies that $\{\rv(\mu_T): T\in \bR^m\}={\rm int}\  \R(\Phi)$ and thus {\rm (ii)} holds.\\
Next, we assume that there exists a nontrivial linear combination of $\Phi$ that is cohomologous to a constant. It is easy to see that in this case $\{\rv(\mu_T): T\in \bR^m\}$ must be contained in some
lower-dimensional affine subspace of $\bR^m$. Since $\overline{\{\rv(\mu_T): T\in \bR^m\}}= \R(\Phi)$ (see \cite[Theorem 1]{Je}), we conclude that $\R(\Phi)$ has empty interior.
\end{proof}
A  side product of the proof of Proposition \ref{proptri} is the following.

\begin{corollary}\label{corintemp}
If ${\rm int}\ \R(\Phi)\not=\emptyset$ then
$\{\rv(\mu_T):\ T\in \bR^m\} = {\rm int}\  \R(\Phi)$.
In particular, $H(w)$ is well approximated by ergodic measures for all $w\in \R(\Phi)$.
\end{corollary}

Next we show that unless $f$ has zero topological entropy the map $w\mapsto H(w)$ is strictly positive on the interior of the rotation set.

\begin{theorem}
Let $f:X\to X$ be a continuous map on a compact metric space satisfying \STP and let $\Phi=(\phi_1,\cdots,\phi_m):X\to\bR^m$ be  H\"older continuous. Then one and only
one of the following conditions holds.
\begin{enumerate}
\item[(i)]
$h_{\rm top}(f)=0$.
\item[(ii)]
 $H(w)> 0$ for all $w\in {\rm int}\  \R(\Phi)$.
\end{enumerate}
\end{theorem}
\begin{proof}
Assume $h_{\rm top}(f)>0$.
If ${\rm int}\ \R(\Phi)=\emptyset$ then (ii) trivially holds. From now on we assume that ${\rm int}\ \R(\Phi)\not=\emptyset$ and thus, by  Proposition \ref{proptri} no nontrivial linear combination of $\Phi$ is cohomologous to a constant.
Let $w_0\in {\rm int}\ \R(\Phi)$. As a consequence of Corollary \ref{corintemp} there exists $T_0\in \bR^m$ such that $\rv(\mu_{T_0})=w_0$. By definition,
$\mu_{T_0}$ is the unique equilibrium measure of $T_0\cdot \Phi$ which implies that $H(w_0)=h(T_0)=h_{\mu_{T_0}}(f)$.

If $T_0=0$, then $\mu_{T_0}$ is the unique measure of maximal entropy and we obtain $H(w_0)=h_{\rm top}(f)>0$.
Otherwise, there exists $w_1\in {\rm int }\ \R(\Phi)$ with
\begin{equation} \label{fgh}
T_0 \cdot w_1 > T_0\cdot w_0.
\end{equation}
It follows  from Corollary \ref{corintemp} that there exists $T_1\in \bR^d$ such that $\rv(\mu_{T_1})=w_1$.
Applying the variational principle and the fact that $\mu_{T_i}$ is the unique equilibrium measure of the potential $T_i\cdot \Phi$ yields
\begin{equation}\label{eqBS}
P_{\rm top}(T_0\cdot \Phi)=  h_{\mu_{T_0}}(f)  +T_0\cdot \rv(\mu_{T_0}) \geq  h_{\mu_{T_1}}(f)+T_0\cdot \rv(\mu_{T_1}).
\end{equation}
Since $T_0\cdot \rv(\mu_{T_i})= T_0\cdot w_i$, \eqref{fgh} and \eqref{eqBS} imply  $H(w_0)=h_{\mu_{T_0}}(f)> h_{\mu_{T_1}}(f)\geq 0$ and we are done.

\end{proof}

Finally we  present the main result of this section.

\begin{theorem}\label{thanalytic}
Let $f:X\to X$ be a continuous map on a compact metric space satisfying \STP and let $\Phi=(\phi_1,\cdots,\phi_m):X\to\bR^m$ be  H\"older continuous. Then $w\mapsto H(w)$ is real-analytic
on $ {\rm int}\  \R(\Phi)$.
\end{theorem}

\begin{proof}
If $ {\rm int}\  \R(\Phi)$ is empty there is nothing to prove. From now on we assume $ {\rm int}\  \R(\Phi)\not=\emptyset$.
Without loss of generality we only consider the case $m=2$ (i.e. $\Phi=(\phi_1,\phi_2)$) and leave the general case to the reader.

Recall that
\begin{equation}
I:\bR^2\to   {\rm int}\  \R(\Phi),\ \ \ T=(t_1,t_2)\mapsto \rv(\mu_T)
\end{equation}
is a real-analytic surjective map. We will show that $I$ is a $C^\omega$-diffeomorphism.
Let $T, S \in \bR^2$ with $\rv(\mu_T)=\rv(\mu_S)$.
Combining that $\mu_{T}$ and $\mu_{S}$ are the unique equilibrium measures of the potentials $T\cdot \Phi$
and $S\cdot \Phi$ respectively, with \eqref{eqvarT} implies $h(T)=h(S)$. We conclude that $\mu_{T}=\mu_{S}$ which implies that
$T\cdot \Phi- S\cdot \Phi$ is cohomologous to a constant. Therefore, $T=S$ and we have shown that $I$ is a bijection. \\
\noindent
Finally, we have to prove that $I$ is a local $C^\omega$-diffeomorphism.
Equations \eqref{gg33} and \eqref{eqw1} combined with the fact that neither $\phi_1$ nor $\phi_2$ are cohomologous to a constant  imply
that
\begin{equation}
\partial_{t_1}^2 Q>0\quad {\rm
 and } \quad \partial_{t_2}^2 Q>0.
\end{equation}
Consider now the bilinear form
\[
A(\varphi_1,\varphi_2)=\partial_{\tau_1}\partial_{\tau_2}
P_{\rm top}(t_1\phi_1+t_2\phi_2+\tau_1\varphi_1+\tau_2\varphi_2)\rvert_{\tau_1=\tau_2=0}.
\]
Then $A(v\phi_1+w\phi_2,v\phi_1+w\phi_2)$ coincides with
\[
\begin{pmatrix}v&w\end{pmatrix}
\begin{pmatrix}A(\phi_1,\phi_2)&A(\phi_1,\phi_2)\\
A(\phi_1,\phi_1)&A(\phi_2,\phi_2)\end{pmatrix}
\begin{pmatrix}v\\w\end{pmatrix}
=
\begin{pmatrix}v&w\end{pmatrix}
B
\begin{pmatrix}v\\w\end{pmatrix},
\]
where
\[
B=\begin{pmatrix} \partial^2_{t_1}Q&
\partial_{t_1}\partial_{t_2}Q\\
\partial_{t_1}\partial_{t_2}Q& \partial^2_{t_2}Q\end{pmatrix}.
\]
Since no nontrivial linear combination of $\phi_1$ and $\phi_1$ is
cohomologous to a constant, if $(v,w)\ne0$ then
$A(v\phi_1+w\phi_2,v\phi_1+w\phi_2)>0$ (see \cite{Ru}) and hence
$B$ is positive definite. In particular $\det B$  is positive.
Using that $B$ is the derivative $DI$ of $I$, implies that ${\rm det} DI(t_1,t_2)>0$ for all $(t_1,t_2)\in \bR^2$. It now follows from the inverse function  theorem  that $I$ is a $C^\omega$-diffeomorphism.
Since $T\mapsto h(T)$ is real-analytic we conclude   that
\[
w\mapsto H(w)=  h\circ I^{-1}(w)
\]
is real-analytic in ${\rm int}\  \R(\Phi)$.

\end{proof}

We say a sequence of compact sets $(E_n)_{n\in \bN}\subset \bR^m$ is a real-analytic exhaustion of $\bR^m$ if $E_n\subset {\rm int }\ E_{n+1}$ for all $n\in \bN$, $\bigcup_{n\in\bN} E_n = \bR^m$ and each $\partial E_n$ is a real-analytic $m-1$-dimensional submanifold of $\bR^m$. The following is a direct consequence of the proof of Theorem \ref{thanalytic}.

\begin{corollary}\label{corende}
Let $f:X\to X$ be a continuous map on a compact metric space satisfying \STP and let $\Phi=(\phi_1,\cdots,\phi_m):X\to \bR^m$ be  H\"older continuous with ${\rm int }\ \R(\Phi)\not=\emptyset. $
Let $(E_n)_{n\in \bN}\subset \bR^m$ is a real-analytic exhaustion of $\bR^m$. Then $C_n=\{\rv(\mu_T): T\in \partial E_n\}\subset {\rm int }\ \R(\Phi)$ is a sequence of $m-1$-dimensional
real-analytic submanifolds that converges to $\partial \R(\Phi)$ in the Hausdorff metric. Moreover, $w\mapsto H(w)$ varies real-analytically on $C_n$ for all $n\in \bN$.
\end{corollary}

\end{document}